\documentclass[11pt]{article}

\usepackage{amsmath,amsfonts,amsthm}
\usepackage[colorlinks,allcolors=blue,pdfencoding=auto]{hyperref}
\usepackage[onehalfspacing]{setspace}
\usepackage{tikz,geometry,enumitem}

\numberwithin{equation}{section}
\setlist[enumerate]{label={\upshape(\roman*)}}
\allowdisplaybreaks

\newcommand	{\N}{\mathbb{N}}
\newcommand	{\R}{\mathbb{R}}
\newcommand	{\E}{\mathbf{E}}
\newcommand	{\PP}{\mathbf{P}} 
\newcommand {\V}{\mathbf{Var}}
\newcommand {\e}{\varepsilon}
\newcommand {\al}{\alpha}
\newcommand {\OO}{\mathcal{O}}
\newcommand {\oo}{{\scriptscriptstyle\mathcal{O}}}
\newcommand {\eq}{\ = \ }
\newcommand {\loe}{\ \leq \ }
\newcommand {\goe}{\ \geq \ }

\newtheorem{definition}{Definition}[section]
\newtheorem{lem}[definition]{Lemma}
\newtheorem{cor}[definition]{Corollary}
\newtheorem{prop}[definition]{Proposition}

\newtheorem*{ex}{Examples}
\newtheorem{theorem}[definition]{Theorem}

\DeclareFontFamily{U}{mathx}{\hyphenchar\font45}
\DeclareFontShape{U}{mathx}{m}{n}{<-> mathx10}{}
\DeclareSymbolFont{mathx}{U}{mathx}{m}{n}
\DeclareFontSubstitution{U}{mathx}{m}{n}
\DeclareMathAccent{\widecheck}{9}{mathx}{"71}

\DeclareFontFamily{U}{mathx}{\hyphenchar\font45}
\DeclareFontShape{U}{mathx}{m}{n}{<-> mathx10}{}
\DeclareSymbolFont{mathx}{U}{mathx}{m}{n}
\DeclareMathAccent{\widebar}{0}{mathx}{"73}

\begin{document}

\title{External branch lengths of $\Lambda$-coalescents\\
without a dust component} 
\author{\noindent Christina S. Diehl\thanks{Institut f\"ur Mathematik, Goethe-Universit\"at, 60054 Frankfurt am Main, Germany \newline diehl@math.uni-frankfurt.de, kersting@math.uni-frankfurt.de \newline Work partially supported by the DFG Priority Programme SPP 1590 ``Probabilistic Structures in Evolution''} $\ $ and G\"otz Kersting$^*$}

\maketitle

\begin{abstract}
\noindent
$\Lambda$-coalescents model genealogies of samples of individuals from a large population by means of a family tree whose branches have lengths. The tree's leaves represent the individuals, and the lengths of the adjacent edges indicate the individuals' time durations up to some common ancestor. These edges are called external branches. 
We consider typical external branches under the broad assumption that the coalescent has no dust component, and maximal external branches under further regularity assumptions. As it transpires, the crucial characteristic is the coalescent's rate of decrease $\mu(b)$, $b\geq 2$. 
The magnitude of a typical external branch is asymptotically given by $n/\mu(n)$, where $n$ denotes the sample size. This result, in addition to the asymptotic independence of several typical external lengths hold in full generality, while convergence in distribution of the scaled external lengths requires that $\mu(n)$ is regularly varying at infinity. 
For the maximal lengths, we distinguish two cases. Firstly, we analyze a class of $\Lambda$-coalescents coming down from infinity and with regularly varying $\mu$. Here the scaled external lengths behave as the maximal values of $n$ i.i.d. random variables, and their limit is captured by a Poisson point process on the positive real line. Secondly, we turn to the Bolthausen-Sznitman coalescent, where the picture changes. Now the limiting behavior of the normalized external lengths is given by a Cox point process, which can be expressed by a randomly shifted Poisson point process.   

\bigskip
\noindent\emph{AMS 2010 subject classification: 60J75 (primary), 60F05, 60J27, 92D25$^{\color{white} \big|}$}\\
\noindent\emph{Keywords:} $\Lambda$-coalescent, dustless coalescent, Bolthausen-Sznitman coalescent, Beta-coales\-cent, Kingman's coalescent, external branch lengths, Poisson point process, Cox point process, weak limit law  
\end{abstract}

\bigskip

\section{Introduction and main results}

In population genetics, family trees stemming from a sample out of a big population are modeled by coalescents. The prominent Kingman coalescent \cite{King82} found widespread applications in biology. More recently, the Bolthausen-Sznitman coalescent, originating from statistical mechanics \cite{BS98}, has gained in importance in analyzing genealogies of populations undergoing selection \cite{BDMM07,DWF13,NH13,Schw17}. Unlike Kingman's coalescent, the Bolthausen-Sznitman coalescent allows multiple mergers. 
The larger class of Beta-coalescents has found increasing interest, e.g., in the study of marine species \cite{SBB13, NNY16}. All these instances are covered by the notion of $\Lambda$-coalescents as introduced by Pitman \cite{Pit99} and Sagitov \cite{Sag99} in 1999. Today, general properties of this extensive class have become more transparent \cite{KSW18,DK18}.  \\
In this paper, we deal with the lengths of external branches of $\Lambda$-coalescents under the broad assumption that the coalescent has no dust component, which applies to all the cases mentioned above. We shall treat external branches of typical and, under additional regularity assumptions, of maximal length. For the total external length, see the publications \cite{Moe10,JK11,DKW14,KPS-J14,DY15}.  

\medskip

$\Lambda$-coalescents are Markov processes $(\Pi(t),\,t\geq 0)$ taking values in the set of partitions of $\N$, where $\Lambda$ denotes a non-vanishing finite measure on the unit interval $[0,1]$.
Its restrictions $(\Pi_n(t),\,t\geq 0)$ to the sets $\{1,\ldots,n\}$ are called $n$-coalescents. They are continuous-time Markov chains characterized by the following dynamics:
Given the event that $\Pi_n(t)$ is a partition consisting of $b\geq 2$ blocks, $k$ specified blocks merge at rate
\[\lambda_{b,k}\ :=\ \int_{[0,1]}p^k(1-p)^{b-k}\frac{\Lambda(dp)}{p^2}, \qquad 2\leq k\leq b, \]
to a single one. 
In this paper, the crucial characteristic of $\Lambda$-coalescents is 
the sequence $\mu=(\mu(b))_{b\geq 2}$ defined as 
\[\mu(b) \ :=\ \sum_{k=2}^b(k-1)\binom{b}{k}\lambda_{b,k}, \qquad b\geq 2.\]
We call this quantity the rate of decrease as it is the rate at which the number of blocks is decreasing on average. Note that a merger of $k$ blocks corresponds to a decline of $k-1$ blocks. 
The importance of $\mu$ also became apparent from other publications \cite{Schw00,LS06,DK18}. In particular, the assumption of absence of a dust component may be expressed in this term. Originally characterized by the condition
\[\int_{[0,1]}\frac{\Lambda(dp)}{p} \eq \infty,\]
(see \cite{Pit99}), 
 \begin{samepage}\enlargethispage{2\baselineskip}
it can be equivalently specified by the requirement
\[\frac{\mu(n)}{n} \ \rightarrow \ \infty\]
as $n\to\infty$  (see Lemma 1 (iii) of \cite{DK18}).
  \end{samepage}  

\medskip

An $n$-coalescent can be thought of as a random rooted tree with $n$ labeled leaves representing the individuals of a sample. 
Its branches specify ancestral lineages of the individuals or their ancestors.  The branch lengths give the time spans until the occurrence of new common ancestors. 
Branches ending in a leaf are called external branches. 
If mutations under the infinite sites model \cite{Kim69} are added in these considerations, the importance of external branches is revealed. This is due to the fact that mutations on external branches only affect a single individual of the sample. Longer external branches result, thereby, in an excess of singleton polymorphisms \cite{WNL-CA01} and are  known to be a characteristic for trees with multiple mergers \cite{EBBF15}; e.g., external branch lengths have been used to discriminate between different coalescents in the context of HIV trees \cite{WBWF16} (see also \cite{VLBRS18}).  Of course such considerations have rather theoretical value as long as singleton polymorphisms cannot be distinguished from sequencing errors. 

\medskip

Now we turn to the main results of this paper. For $1\leq i\leq n$, the length of the external branch ending in leaf $i$ within an $n$-coalescent is defined as 
\[T_i^n\ :=\ \inf{\left\{t\geq 0:\;\left\{i\right\}\notin\Pi_{n}(t)\right\}}.\]
In the first theorem, we consider the length $T^n$ of a randomly chosen external branch. Based on the exchangeability, $T^n$ is equal in distribution to $T_i^n$ for $1\leq i\leq n$.  
The result clarifies the magnitude of $T^n$ in full generality. 
 
\begin{theorem}
\label{dustless}
For a $\Lambda$-coalescent without a dust component, we have for $t\geq 0$, 
\[e^{-2t}+\oo(1)\ \leq \ \PP\left(\frac{\mu(n)}{n}\ T^{n}> t\right)\ \leq \ \frac{1}{1+t}+\oo(1)\]
as $n\to\infty$.
\end{theorem}

\medskip

Among others, this theorem excludes the possibility that $T^n$ converges to a positive constant in probability. In \cite{KSW14} the order of $T^n$ was interpreted as the duration of a generation, namely the time at which a specific lineage, out of the $n$ present ones, takes part in a merging event. In that paper, only Beta$(2-\al,\al)$-coalescents with $1<\al<2$ were considered, and the duration was given as $n^{1-\al}$. Our theorem shows that for this quantity the term $n/\mu(n)$ is a suitable measure for $\Lambda$-coalescents without a dust component.   

\medskip

Asymptotic independence of the external branch lengths holds as well in full generality for dustless coalescents. In light of the waiting times, which the different external branches have in common, this may be an unexpected result. However, this dependence vanishes in the limit. Then it becomes crucial whether two external branches end in the same merger. Such an event is asymptotically negligible only in the dustless case. This heuristic motivates the following result.

\begin{theorem}
\label{indep}
A $\Lambda$-coalescent has no dust component if and only if for fixed $k\in\N $ and for any sequence of numbers $t_1^n,\, \ldots,\, t_k^n\geq 0$, $n\geq 2$, we have
\[\PP\left(T_1^n\,\leq\, t_1^n,\,\ldots,\,T_k^n\,\leq\, t_k^n\right)\eq \PP\left(T_1^n\,\leq\, t_1^n\right)\,\cdots\,\PP\left(T_k^n\,\leq\, t_k^n\right)\,+\,\oo(1)\]
as $n\to\infty$.
\end{theorem}

\medskip

In the dustless case, one has $T_i^n \to 0$ in probability for $1\leq i\leq k$, then one reasonably restricts to the case $t_i^n\to 0$  as $n\to\infty$. 
The statement that the asymptotic independence fails for coalescents with a dust component goes back to M{\"o}hle  (see equation (10) of \cite{Moe10}). 

\medskip

In order to achieve convergence in distribution of the scaled lengths, stronger assumptions are required on the rate of decrease, namely that $\mu$ is a regularly varying sequence. A characterization of this property is given in Proposition \ref{reg_coa} below. Let $\delta_0$ denote the Dirac measure at zero. 

\begin{theorem}
\label{iff}
For a $\Lambda$-coalescent without a dust component, there is a sequence $(\gamma_n)_{n\in\N}$ such that $\gamma_n\,T^n$ converges in distribution to a probability measure unequal to $\delta_0$ as $n\to\infty$ if and only if $\mu$ is regularly varying at infinity.  
Then its exponent $\alpha$ of regular variation fulfills $1\leq\alpha\leq 2$ and we have
\begin{enumerate}
\item for $1<\alpha\leq 2$,
\begin{align*}
\PP\left(\frac{\mu(n)}{n}\ T^n> t\right) \ \longrightarrow \ \frac{1}{\left(1+\left(\alpha-1\right)t\right)^{\frac{\alpha}{\alpha-1}}} \,, \qquad t\geq 0,
\end{align*}
\item for $\alpha=1$,
\begin{align*}
\PP\left(\frac{\mu(n)}{n}\ T^n> t\right) \ \longrightarrow \ e^{-t}, \qquad t\geq 0,
\end{align*}
\end{enumerate}
as $n\to\infty$.
\end{theorem}

\medskip

In particular, this theorem includes the special cases known from the literature. Blum and Fran\c{c}ois \cite{BF05}, as well as Caliebe et al. \cite{CNKR07}, studied Kingman's coalescent.  
For the Bolthausen-Sznitman coalescent, Freund and M{\"o}hle \cite{FM09} showed asymptotic exponentiality of the external branch length. This result was generalized by Yuan \cite{Y14}. 
A class of coalescents containing the Beta$(2-\alpha,\alpha)$-coalescent with $1<\alpha<2$ was analyzed by Dhersin et al. \cite{DFS-JY13}. 

\medskip

Combining Theorem \ref{indep} and \ref{iff} yields the following corollary:

\begin{samepage}\enlargethispage{2\baselineskip}
\begin{cor}
\label{cor}
Suppose that the $\Lambda$-coalescent lacks a dust component and has regularly varying rate of decrease $\mu$ with exponent $\alpha\in\left[1,2\right]$. Then for fixed $k\in\N $, we have 
\[\frac{\mu(n)}{n}\,\left(T_1^n,\,\ldots,\,T_k^n\right)\ \stackrel{d}{\longrightarrow} \ \left(T_1,\,\ldots,\,T_k\right)\]
as $n\to\infty$, where $T_1,\,\ldots,\,T_k$ are i.i.d. random variables each having the density
\begin{align} \label{dens}
f(t)\,dt \eq \frac{\alpha}{\left(1+\left(\alpha-1\right)t\right)^{1+\frac{\alpha}{\alpha-1}}}\ dt \,, \quad t\geq 0,
\end{align}
for $1<\alpha\leq 2$ and a standard exponential distribution for $\alpha=1$. 
\end{cor}
 \end{samepage}

\medskip



\begin{ex} 
For $k\in\N$, let $T_1,\,\ldots,\,T_k$ be the i.i.d. random variables from Corollary \ref{cor}. 
\begin{enumerate}
\item If $\Lambda\left(\left\{0\right\}\right)=2$, then $\mu(n)\sim n^2$ and consequently
\[n\left(T_1^n,\ldots,T_k^n\right)\ \stackrel{d}{\longrightarrow} \ \left(T_1,\ldots,T_k\right)\]  
as $n\to\infty$. 
This statement covers (after scaling) the Kingman case. Note that $\Lambda|_{(0,1]}$ does not affect the limit. 
\item 
If $\Lambda(dp)= c_a\,p^{a-1}(1-p)^{b-1}dp$ for $0<a<1$, $b>0$ and $c_a:=(1-a)(2-a)/\Gamma(a)$, then $\mu(n)\sim n^{2-a}$   
and therefore
\[n^{1-a}\left(T_1^n,\ldots,T_k^n\right)\ \stackrel{d}{\longrightarrow} \ \left(T_1,\ldots,T_k\right)\]
as $n\to\infty$.  
After scaling, this includes the Beta$(2-\alpha,\alpha)$-coalescent with $1<\alpha<2$ (see Theorem 1.1 of Siri-J{\'e}gousse and Yuan \cite{SY16}). Note that the constant $b$ does not appear in the limit. 
\item 
If $\Lambda(dp)=(1-p)^{b-1}dp$ with $b>0$, then we have $\mu(n)\sim n\log{n}$ implying
\begin{align}\label{ex}
\log{n}\left(T_1^n,\ldots,T_k^n\right)\ \stackrel{d}{\longrightarrow} \ \left(T_1,\ldots,T_k\right)
\end{align}
as $n\to\infty$. 
This contains the Bolthausen-Sznitman coalescent (see Corollary 1.7 of Dhersin and M{\"o}hle \cite{DM13}). Again the constant $b$ does not show up in the limit. 
\end{enumerate}
\end{ex}

\bigskip

In the second part of this paper, we change perspective and examine the external branch lengths ordered by size downwards from their maximal value. In this context, an approach via a point process description is appropriate. 
Here we consider $\Lambda$-coalescents having regularly varying rate of decrease $\mu$, additionally to the absence of a dust component.  
It turns out that one has to distinguish between two cases.   

\medskip

First, we treat the case of $\mu$ being regularly varying with exponent $\al\in(1,2]$ (implying that the coalescent comes down from infinity).   
We introduce the sequence $(s_n)_{n\geq 2}$ given by
\begin{equation}\label{new2}
\mu(s_n) \eq \frac{\mu(n)}{n}\,.
\end{equation}
Note that $\mu(n)/n$ is a strictly increasing and, in the dustless case, diverging sequence (see Lemma~\ref{prop}~(ii) and (iv) below), which directly transfers to the sequence $(s_n)_{n\geq 2}$. 
Also note in view of Lemma \ref{prop} (ii) below that 
 \begin{samepage}\enlargethispage{2\baselineskip}
\begin{align}\label{s_n}
s_n \eq \oo(n)
\end{align}
as $n\to\infty$.
 \end{samepage}

\begin{ex}
\begin{enumerate}
\item If $\mu(n)\sim n^\alpha$ with $\al\in(1,2]$, then we have $s_n\sim n^{(\al-1)/\al}$ as $n\to\infty$. 
\item 
If $\mu$ is regularly varying with exponent $\al\in(1,2]$, then the sequence $s_n$ is regularly varying with exponent $(\al-1)/\al$. 
\end{enumerate}
\end{ex}

\medskip
 
We define point processes $\Phi^{\,n}$ on $\left(0,\infty\right)$ via 
\[\Phi^{\,n}(B)\ :=\ \#\left\{i\leq n: \;\frac{\mu(n)}{ns_n}\ T_i^{n}\in B\right\}\]
for Borel sets $B\subset\left(0,\infty\right)$.  

\begin{theorem} \label{reg_var}  
Assume that the $\Lambda$-coalescent has a regularly varying rate of decrease $\mu$ with exponent $\al\in(1,2]$. 
Then, as $n\rightarrow\infty$, the point process $\Phi^{\,n}$ converges in distribution to a Poisson point process $\Phi$ on $(0,\infty)$ with intensity measure 
\[\phi(dx) \eq \frac{\alpha }{\left(\left(\alpha-1\right)x\right)^{1+\frac{\al}{\al-1}}}\ dx.\]
\end{theorem}

\medskip

Note that $\int_0^1\phi(x)dx=\infty$, which means that the points from the limit $\Phi$ accumulate at the origin. On the other hand, we have $\int_1^\infty\phi(x)dx<\infty$ saying that the points can be arranged in decreasing order. Thus, the theorem focuses on the maximal external lengths showing that the longest external branches differ from a typical one by the factor $s_n$ in order of magnitude (see Corollary \ref{cor}). 
For Kingman's coalescent, this result was obtained by Janson and Kersting \cite{JK11} using a different method. 

\medskip

In particular, letting $T_{\left\langle 1\right\rangle}^n$ be the maximal length of the external branches, we obtain for $x>0$,
\[\PP\left(\frac{\mu(n)}{ns_n}T_{\left\langle 1\right\rangle}^n \loe x\right) \ \to \ e^{-((\alpha-1)x)^{-\frac{\alpha}{\alpha-1}}}\]
as $n\to\infty$, i.e., the properly scaled $T_{\left\langle 1\right\rangle}^n$ is asymptotically Fréchet-distributed. 

\medskip

Corollary \ref{cor} shows that the external branch lengths behave for large $n$ as i.i.d. random variables. This observation is emphasized by Theorem \ref{reg_var} because the maximal values of i.i.d. random variables, with the densities stated in Corollary \ref{cor}, have the exact limiting behavior as given in Theorem \ref{reg_var} (including the scaling constants $s_n$). 

\medskip

This heuristic fails for the Bolthausen-Sznitman coalescent, which we  now address.  
For  $n\in\N$, define the quantity  
\[t_n\ :=\ \log\log{n}-\log\log\log{n}+\frac{\log\log\log{n}}{\log\log{n}},\]
where we put $t_n:=0$ if the right-hand side is negative or not well-defined. 
Here we consider the point processes $\Psi^{\,n}$ on the whole real line given by  
\[\Psi^{\,n}(B)\ :=\ \#\left\{i\leq n:\, \log{\log{(n)}}(T_i^n-t_n)\in B\right\}\]
for Borel sets $B\subset\R $. As before, we focus on the maximal values of $\Psi^n$.  

\begin{theorem}\label{bs}
For the Bolthausen-Sznitman coalescent, the point process $\Psi^{\, n}$ converges in distribution as $n\to\infty$ to a Cox point process $\Psi$ on $\R $ directed by the random measure
\[\psi\left(dx\right) \eq E\,e^{-x} dx,\] 
where $E$ denotes a standard exponential random variable. 
\end{theorem}

\medskip

Observe that this random density may be rewritten as
\[e^{-x+\log{E}}dx.\]
This means that the limiting point process can also be considered as a Poisson point process with intensity measure $e^{-x}dx$ shifted by the independent amount  $\log{E}$. This alternative representation will be used in the theorem's proof (see Theorem \ref{bs v2} below). Recall that $G:=-\log{E}$ has a standard Gumbel distribution. 

\medskip

In particular, letting again $T_{\left\langle 1\right\rangle}^n$ be the maximum of $T_1^n, \ldots, T_n^n$, we obtain
\begin{equation}\label{new}
\PP\left(\log{\log{(n)}}(T_{\left\langle 1\right\rangle}^n -t_n)\loe x\right) \ \longrightarrow \  \int_0^\infty e^{-ye^{-x}}e^{-y}\; dy \eq \frac{1}{1+e^{-x}}
\end{equation}
as $n\to\infty$. Notably, we arrive at a limit that is non-standard in the extreme value theory of i.i.d. random variables, namely the so-called logistic distribution. 

\medskip

We point out that the limiting point process $\Psi$ no longer coincides with the limiting Poisson point process as obtained for the maximal values of $n$ independent exponential random variables. The same turns out to be true for the scaling sequences. In order to explain these findings, 
note that \eqref{new} implies
\[\frac{T^n_{\left\langle 1\right\rangle}}{\log\log{n}} \eq 1 +\oo_p(1)\] 
as $n\to\infty$, where  $\oo_p(1)$ denotes a sequence of random variables converging to $0$ in probability.
In particular, $T^n_{\left\langle 1\right\rangle}\to\infty$ in probability. 
Hence, we pass with this theorem to the situation where very large mergers affect the maximal external lengths. 
Then circumstances change and new techniques are required. For this reason, we have to confine ourselves to the Bolthausen-Sznitman coalescent in the case of regularly varying $\mu$ with exponent $\al=1$.   
  
\medskip

It is interesting to note that an asymptotic shift by a Gumbel distributed variable also shows up in the absorption time $\widetilde{\tau}_n$ (the moment of the most recent common ancestor) of the Bolthausen-Sznitman coalescent:
\[\widetilde{\tau}_n-\log{\log{n}} \ \stackrel{d}{\longrightarrow}\ G\]
as $n\to\infty$ (see Goldschmidt and Martin \cite{GM05}).   
However, this shift remains unscaled. Apparently, these two Gumbel distributed variables under consideration build up within different parts of the coalescent tree.   

\bigskip

Before closing this introduction, we provide some hints concerning the proofs. For the first three theorems, we make use of an asymptotic representation for the tail probabilities of the external branch lengths. 
Remarkably, this representation involves, solely, the rate of decrease $\mu$, though in a somewhat implicit, twofold manner. 
The proofs of the three theorems consist in working out the consequences of these circumstances.  
The representation is given in Theorem \ref{thm:int} and relies largely on different approximation formulae derived in \cite{DK18}. We recall the required statements in Section \ref{SLOLN}. \medskip

The proofs of the last two theorems incorporate Corollary \ref{cor} as one ingredient.  
The idea is to implement stopping times $\widetilde{\rho}_{c,n}$ with the property that at that moment a positive number of external branches is still extant which is of order 1 uniformly in $n$. 
To these remaining branches, the results of Corollary \ref{cor} are applied taking the strong Markov property into account.  
More precisely, let 
\[N_n \eq \left(N_n(t), t\geq 0\right) \] 
be the block counting process of the $n$-coalescent, where 
\[N_n(t) \ := \ \#\Pi_n(t)\] 
states the number of lineages present at time $t\geq 0$. For definiteness, we put $N_n(t)=1$ for $t>\widetilde{\tau}_n$.  
In the case of regularly varying $\mu$ with exponent $1<\al\leq 2$, we will show that 
\[\widetilde{\rho}_{c,n}\ :=\ \inf\left\{t\geq 0:\,N_n(t)\leq c s_n\right\}\]
with arbitrary $c>0$ is the right choice. 
Next, we split the external lengths $T^n_{i}$ into  
the times $\widecheck{T}^{n}_{i}$ up to the moment $\widetilde{\rho}_{c,n}$ and the residual times $\widehat{T}^{n}_{i}$. 
Formally, we have 
\[\widecheck{T}^{n}_{i}\ :=\ T^n_{i}\wedge\widetilde{\rho}_{c,n} \qquad \text{ and }\qquad \widehat{T}^{n}_i\ :=\ T^n_{i}-\widecheck{T}^{n}_{i}.\] 
We shall see that $\widecheck{T}^{n}_{i}$ is of negligible size compared to $\widehat{T}^{n}_{i}$ for large values of $c$. On the other hand, with increasing $c$, also the number of extant external branches tends to infinity uniformly in $n$. Corollary~\ref{cor} tells us that the $\widehat{T}^{n}_{i}$ behave approximately like i.i.d. random variables. Therefore, one expects that the classical extreme value theory applies in our context. These are the ingredients of the proof of Theorem~
\ref{reg_var}. 
\bigskip

\begin{figure}[h]
\center{\includegraphics[width=0.8\textwidth]{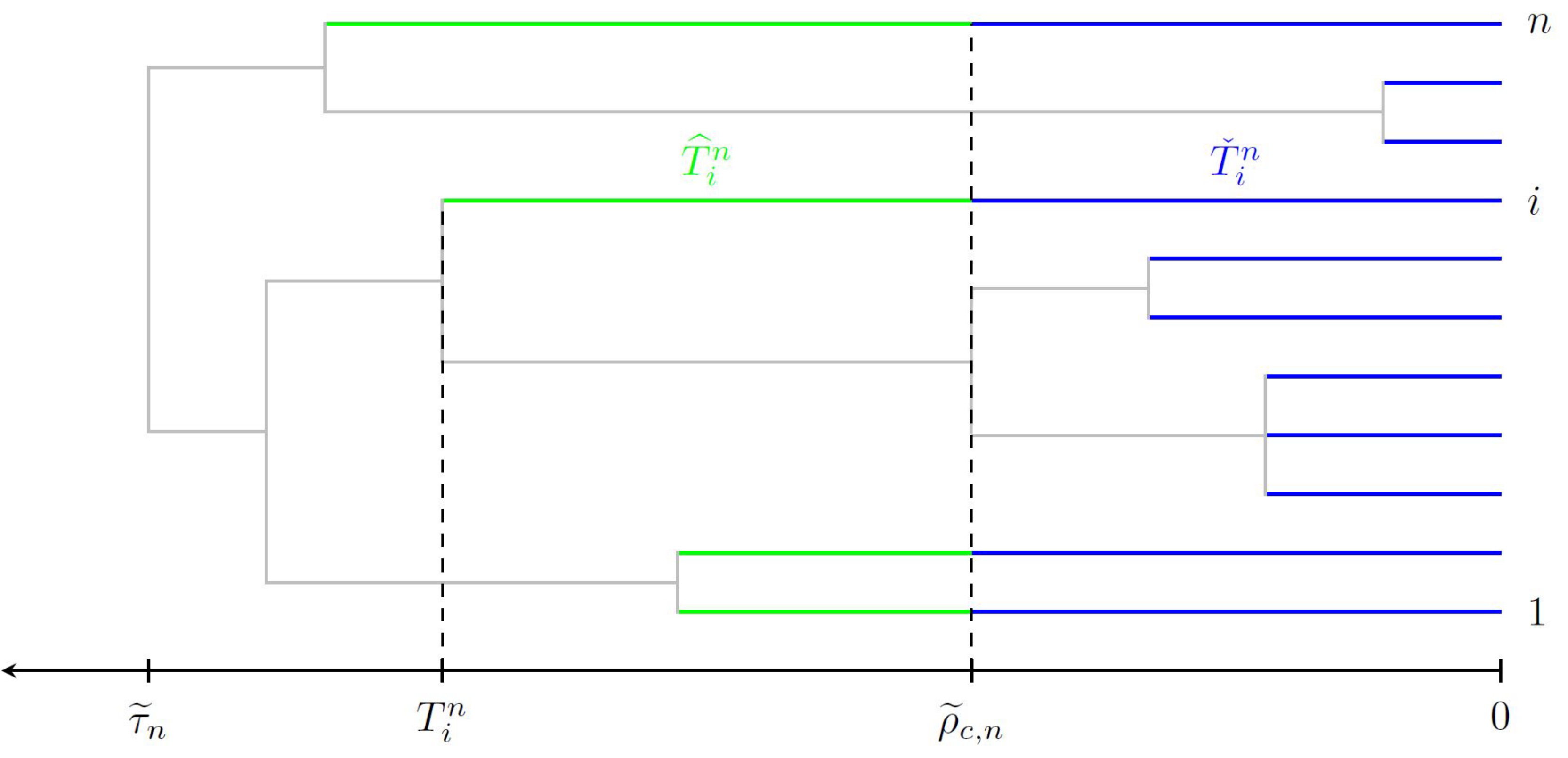}}
\caption{The stopping time $\widetilde{\rho}_{c,n}$ subdividing the external branch ending in leaf $i$ into two parts of length $\widecheck{T}_i^n$ and $\widehat{T}_i^n$, respectively.}
\end{figure}

\medskip

The approach for the Bolthausen-Sznitman coalescent is essentially the same.  However, new obstacles appear. In contrast to the previous case $\al>1$, the lengths of the maximal branches now diverge in probability. As a consequence, in the case $\alpha=1$, we have in general no longer control over the stopping times $\widetilde{\rho}_{c,n}$ as defined above. 
Fortunately, for the Bolthausen-Sznitman coalescent, Möhle \cite{Moe15} provides a precise asymptotic description of the block counting process $N_n$ by means of the Mittag-Leffler process, which applies also in the large time regime. 
Adapted to this result, the role of $\widetilde{\rho}_{c,n}$ is taken by $t_{c,n}\wedge\widetilde{\tau}_n$, where
\[t_{c,n}\ :=\ t_n-\frac{\log{c}}{\log\log{n}}\] 
for some $c>1$. 
Thus, for the Bolthausen-Sznitman coalescent, the external lengths $T^n_{i}$ are split into 
\[\widecheck{T}^{n}_{i}\ :=\ T^n_{i}\wedge t_{c,n} \qquad \text{ and }\qquad \widehat{T}^{n}_i\ :=\ T^n_{i}-\widecheck{T}^{n}_{i}.\]

In contrast to the case $\al>1$, the part $\widecheck{T}^{n}_{i}$ does not disappear for $c\to\infty$ but is asymptotically Gumbel-distributed and shows up in the above mentioned independent shift.  

\bigskip

The paper is organized as follows: In Section \ref{SLOLN} we recapitulate some laws of large numbers from \cite{DK18}. Section \ref{sec_prop} summarizes several properties of the rate of decrease. The fundamental asymptotic expression of the external tail properties is developed in Section~\ref{sec_rand}. Sections \ref{sec_proofs} and \ref{sec_proof} contain the proofs of Theorem \ref{dustless} to \ref{iff}. In Section \ref{sec_mom} we prepare the proofs of the remaining theorems by establishing a formula for factorial moments of the number of external branches. Sections \ref{sec_proof_beta} and \ref{sec_proof_bs} include the proofs of Theorem \ref{reg_var} and \ref{bs}.

\bigskip

\section{Some laws of large numbers}\label{SLOLN}

In this section we report on some laws of large numbers from the recent publication \cite{DK18}, which are a main tool in the subsequent proofs. Let $X=(X_j)_{j\in\N_0}$ denote the Markov chain embedded in the block-counting process $N_n$, i.e., $X_j$ denotes the number of branches after $j$ merging events. (For convenience, we suppress $n$ in the notation of $X$.) Also, let
\[ \rho_r :=\min \{ j\ge 0: X_j \le r\} \]
 for numbers $r>0$.  We are dealing with laws of large numbers  for  functionals of the form
\[ \sum_{j=0}^{\rho_{r_n}-1} f(X_j)  \]
with some suitable positive function $f$ and some sequence $(r_n)_{n \ge 1}$  of positive numbers.
These laws of large numbers  build on two approximation steps. First, letting
\[ \Delta X_{j+1}:= X_{j}-X_{j+1} \quad \text{and}\quad \nu(b):= \mathbb E[\Delta X_{j+1} \mid X_{j}=b] \]
for $j \ge 1$, we notice that for large $n$, 
\[ \sum_{j=0}^{\rho_{r}-1}f(X_j) \approx  \sum_{j=0}^{\rho_{r}-1} f(X_j)\frac {\Delta X_{j+1}}{\nu(X_j)} .\]
The rationale of this approximation consists in the observation that the difference of both sums stems from the martingale difference sequence $f(X_j)(1-\Delta X_{j+1}/\nu(X_j))$, $j \ge 0$, and, thus, is of a comparatively  negligible order. Second, we  remark that 
\[  \sum_{j=0}^{\rho_{r}-1}\frac { f(X_j)}{\nu(X_j)}\Delta X_{j+1} \approx \int_{r}^n \frac {f(x)}{\nu(x)}\, dx,\]
with $\nu(x)$ extending the numbers $\nu(b)$ to  real numbers $x \ge 2$. Here, we regard the left-hand sum as a Riemann approximation of the right-hand integral and take $X_{\rho_{r}} \approx r$ into account. Altogether,
\[\sum_{i=0}^{\rho_{r}-1} f(X_i)  \approx \int_{r}^n \frac {f(x)}{\nu(x)} .\]

\medskip

In order to estimate the errors and, in particular, the martingale's quadratic variation, different assumptions are required. For details we refer to \cite{DK18} and deal here only with the two cases that we use later in our proofs.

\medskip

The first case concerns the time 
\[\widetilde \rho_{r}:= \inf\{ t\ge 0: N_n(t)\loe r\},\] 
when the block-counting process drops below $r$. Letting $W_j$ be the period of stay of $N_n$ at state $X_j$ (again suppressing $n$ in the notation), we have 
\[ \widetilde \rho_{r} = \sum_{j=0}^{\rho_{r}-1} W_j \approx \sum_{j=0}^{\rho_r-1}\mathbb E[W_j \mid N_n]=\sum_{j=0}^{\rho_{r}-1} \frac 1{\lambda(X_j)},\]
where $\lambda(b):= \sum_{2\le k \le b} \lambda_{b,k}$ is the jump rate of the block counting process. Also, $\nu(b)= \mu(b)/\lambda(b)$. Therefore, putting $f(x)=\lambda(x)^{-1}$, we are led to the approximation formula
\[  \rho_{r} \approx \int_r^n \frac{dx}{\mu(x)} \ . \]
More precisely, we have the following law of large numbers.

\begin{prop} \label{ErgLem1}
Assume that the $\Lambda$-coalescent is dustless. 
Let $\gamma < 1$ and let $2 \le r_n \le \gamma n$, $n \ge 1$, be numbers such that
\[ \int_{r_n}^n \frac{dx}{\mu(x)} \to 0 \]
as $n \to \infty$. Then
\[ \widetilde \rho_{r_n} = (1+ \oo_P(1)) \int_{r_n}^n \frac{dx}{\mu(x)} \]
as $n\to \infty$.
\end{prop}

\newpage

The role of the assumptions is easily understood: The condition $\int_{r_n}^n \frac{dx}{\mu(x)} \to 0$ implies that $\tilde \rho_{r_n}\to 0$ in probability, i.e., we are in the small time regime. This is required to avoid very large jumps $\Delta X_{j+1}$ of order $X_{j+1}$, which  would ruin the above Riemann approximation. The condition $r_n \le \gamma n$ guarantees that $\widetilde \rho_{r_n}$ is sufficiently large to allow for a law of large numbers.

\medskip

Secondly, we turn to the case $f(x)=x^{-1}$. Here we point out that as $x \to \infty$,
\[ \frac 1{\nu(x)} \sim x \frac d{dx} \log \frac{\mu(x)}x, \]
which follows from \cite[Lemma 1 (ii)]{DK18}. Therefore,
\[ \int_r^n \frac{dx}{x\nu(x)} \approx \log \Big(\frac{\mu(n) }{n}\frac r{\mu(r)}\Big)\,, \]
and we have the following law of large numbers.

\begin{prop}\label{ErgLem2}
Under the assumptions of the previous proposition, we have
\[ \sum_{j=0}^{\rho_{r_n}-1} \frac 1{X_j} = (1+ \oo_P(1))\log \Big(\frac{\mu(n) }{n}\frac {r_n}{\mu(r_n)}\Big) \quad \text{and} \quad \sum_{j=0}^{\rho_{r_n}-1} \frac 1{X_j} = \log \Big(\frac{\mu(n) }{n}\frac {r_n}{\mu(r_n)}\Big) + \oo_P(1) \]
as $n \to \infty$.
\end{prop}

For the proofs of these propositions, see \cite[Section 3]{DK18}.

\bigskip

\section{Properties of the rate of decrease}\label{sec_prop}
We now have a closer look at the rate of decrease $\mu$ introduced in the first section. Defining
\begin{align}\label{mu_def}
\mu(x) \ :=\ \int_{[0,1]}\left(xp-1+(1-p)^x\right)\frac{\Lambda(dp)}{p^2}, 
\end{align}
we extent $\mu$ to all real values $x\geq 1$, where the integrand's value at $p=0$ is understood to be $x(x-1)/2$. 

\medskip

The next lemma summarizes some required properties of $\mu$.

\begin{lem} \label{prop}
The rate of decrease and its derivatives have the following properties:
\begin{enumerate}
\item $\mu(x)$ has derivatives of any order with finite values, also at $x=1$. Moreover, $\mu$ and $\mu'$ are both non-negative and strictly increasing, while $\mu''$ is a non-negative and decreasing function. 
\item For $1<x\leq y$, 
\begin{align*}
\frac{x(x-1)}{y(y-1)}\loe\frac{\mu(x)}{\mu(y)}\loe \frac{x}{y}\,.
\end{align*}
\item For $x> 1$, 
\begin{align*}
\mu'(1) \loe \frac{\mu(x)}{x-1} \loe \mu'(x)\qquad \text{ and }\qquad  \mu''(x) \loe \frac{\mu'(x)}{x-1}.
\end{align*}  
\item In the dustless case, 
\[\frac{\mu(x)}{x} \ \rightarrow \ \infty \]
as $x\to\infty$. 
\end{enumerate}
\end{lem}

\begin{proof}
(i) Let 
\[\mu_2(x) \ := \ \int_{[0,1]}(1-p)^x\log^2{(1-p)}\frac{\Lambda(dp)}{p^2},\]
which is a $\mathcal{C}^\infty$-function for $x>0$. Set
\begin{align*}
\mu_1(x)\ := \ & \int_1^x \mu_2(y)dy + \int_{[0,1]}\left(p+(1-p)\log{(1-p)}\right)\frac{\Lambda(dp)}{p^2}\\
\eq &  \int_{[0,1]}\left((1-p)^x\log{(1-p)}+p\right)\frac{\Lambda(dp)}{p^2}\,.
\end{align*}
Note that the second integral in the first line is finite and non-negative just as its integrand. Then we have
\[\mu(x) \eq \int_{1}^x\mu_1(y)dy.\]
Thus, $\mu_1(x)=\mu'(x)$ and $\mu_2(x)=\mu''(x)$ for $x\geq 1$. From these formulae our claim follows.  

\medskip

(ii) The inequalities are equivalent to the fact that $\mu(x)/x$ is increasing and $\mu(x)/(x(x-1))$ is decreasing as follows from formulae (7) and (8) of \cite{DK18}. \medskip

(iii) The monotonicity properties from (i) and $\mu(1)=0$ yield for $x\geq 1$, 
\[\mu'(1)(x-1) \loe \mu(1) + \int_1^x\mu'(y)dy \loe \mu'(x)(x-1).\]
Similarly, we get $\mu''(x)(x-1)\leq\mu'(x)$ because $\mu'(1)\geq 0$. \medskip

(iv) See Lemma 1 (iii) of \cite{DK18}. 
\end{proof}

\bigskip

In order to characterize regular variation of $\mu$, we introduce the function
\[H(u)\ :=\ \frac{\Lambda{\left(\left\{0\right\}\right)}}{2}+\int_0^uh(z)dz\,, \quad 0\leq u\leq 1,\]
where
\[h(z)\ :=\ \int_z^1\int_{\left(y,1\right]}\frac{\Lambda\left(dp\right)}{p^2}dy\,, \quad 0\leq z\leq 1.\]
Note that $H$ is a finite function because we have
\begin{align}\label{H(1)}
H(1)\eq\frac{\Lambda{\left(\left\{0\right\}\right)}}{2}+\int_0^1 \int_0^p\int_0^y dz\,dy\,\frac{\Lambda(dp)}{p^2} \eq \frac{\Lambda\left(\left[0,1\right]\right)}{2}\ <\ \infty\,.
\end{align}

\medskip

\begin{prop}\label{reg_coa} 
For a $\Lambda$-coalescent without a dust component, the following statements hold: 
\begin{enumerate}
\item
$\mu(x)$ is regularly varying at infinity if and only if $H(u)$ is regularly varying at the origin. 
Then $\mu$ has an exponent $\alpha\in[1,2]$ and we have 
\begin{equation}\label{H}
\mu(x)\ \sim\ \Gamma(3-\alpha)\,x^{2}\,H\left(x^{-1}\right)
\end{equation}
as $x\to\infty$. 
\item $\mu(x)$ is regularly varying at infinity with some exponent $\alpha\in(1,2)$ if and only if the function $\int_{(y,1]}p^{-2}\Lambda(dp)$ is regularly varying at the origin with an exponent $\alpha\in(1,2)$. Then we have
\[\mu(x)\ \sim\ \frac{\Gamma(2-\alpha)}{\alpha-1} \int_{x^{-1}}^1\frac{\Lambda(dp)}{p^2}\]
as $x\to\infty$.
\end{enumerate}
\end{prop}

The last statement brings the regular variation of $\mu$ together with the notion of regularly varying $\Lambda$-coalescents as introduced in \cite{DK18}.
  
\medskip

For the proof of this proposition, we apply the following characterization of regular variation. 

\begin{lem}\label{deriv}
Let $V(z)$, $z>0$, be a positive function with an ultimately monotone derivative $v(z)$ and let $\eta\neq 0$. 
Then $V$ is regularly varying at the origin with exponent $\eta$ if and only if $\left|v\right|$ is regularly varying at the origin with exponent $\eta-1$ and 
\begin{align*}
z\,v(z)\ \sim \ \eta\, V(z) 
\end{align*}
as $z\to 0^+$.
\end{lem}

\newpage

\begin{proof}
For $\eta>0$, we have $V(0+)=0$ and, therefore, $V(z)=\int_0^zv(y)dy$. For $\eta<0$, we use the equation $V(z)=\int_z^1(-v(y))dy+V(1)$ instead: here it holds $V(0+)=\infty$. Now our claim follows from well known results for regularly varying functions at infinity (see \cite{Sen73} as well as Theorem~1~(a) and (b) in Section VIII.9 \cite{Fell71}).
The proofs translate one-to-one to regularly varying functions at the origin. 
\end{proof}

\bigskip

\begin{proof}[Proof of Proposition \ref{reg_coa}] 
(i) From the definition \eqref{mu_def}, we obtain by double partial integration (see formula (8) of \cite{DK18}) that
\begin{align}\label{mu}
\frac{\mu(x)}{x(x-1)}\eq\frac{\Lambda(\left\{0\right\})}{2}+\int_0^1 (1-z)^{x-2}h(z)
\, dz.
\end{align}

If $\Lambda(\{0\})>0$, then our claim is obvious because the first term of the right-hand side of \eqref{mu} dominates the integral as $x\rightarrow\infty$ implying $\mu(x)/x^2\sim\Lambda(\{0\})/2=H(0)$ and, therefore, $\alpha=2$. Thus, let us assume that $\Lambda(\{0\})=0$. 
Let 
\begin{align*}
\mathcal{L}(x)\ :=\ \int_0^1 e^{-zx} h(z)\, dz
\end{align*}
be the Laplace transform of $H$.    
In view of a Tauberian theorem (see Theorem~3 and Theorem~2 in Section XIII.5 of \cite{Fell71}), it is sufficient to prove that
\begin{align}\label{int_sim}
\mathcal{L}(x)\ \sim\ \frac{\mu(x)}{x^2}
\end{align}
as $x\rightarrow\infty$. 
For $\frac{1}{2}<\delta<1$, let us consider the decomposition
\begin{align}\label{sum_int}
\frac{\mu(x)}{x(x-1)} \eq \int_0^{x^{-\delta}} (1-z)^{x-2} h(z)\, dz+\int_{x^{-\delta}}^1 (1-z)^{x-2} h(z)\, dz.
\end{align}
Because of $\delta<1$ and \eqref{H(1)}, we have
\begin{align}
\int_{x^{-\delta}}^1 (1-z)^{x-2} h(z)\, dz \loe  (1-x^{-\delta})^{x-2} \int_{x^{-\delta}}^1h(z)\, dz\loe e^{-x^{-\delta}(x-2)}\,H(1) \eq\oo\left(x^{-1}\right)
\end{align}
as $x\rightarrow\infty$. In particular, the second integral in the decomposition \eqref{sum_int} can be neglected in the limit $x\rightarrow\infty$ since $\mu(x)/(x(x-1))\geq \mu'(1)/x$ due to Lemma \ref{prop} (iii).  
As to the first integral in \eqref{sum_int}, observe for $\delta>\frac{1}{2}$ that 
\[-\log{\frac{(1-z)^{x-2}}{e^{-zx}}}\eq \OO(x^{1-2\delta})\ \longrightarrow \ 0\]
uniformly for $z\in[0,x^{-\delta}]$ as $x\to\infty$ and, therefore,
\begin{align}\label{int1}
\int_0^{x^{-\delta}} (1-z)^{x-2} h(z)\, dz\ \sim\ \int_0^{x^{-\delta}} e^{-zx} h(z)\, dz.
\end{align}
Also note that
\begin{align}\label{int2}
\int_{x^{-\delta}}^1e^{-zx}h(z)dz\loe e^{-x^{1-\delta}}H(1) 
\eq\oo\left(x^{-1}\right)
\end{align}
as $x\to\infty$. Combining \eqref{sum_int} to \eqref{int2} entails 
\[\int_0^1 (1-z)^{x-2} h(z) dz\ \sim \ \mathcal{L}(x).\]
Hence, along with formula \eqref{mu}, this proves the asymptotics in \eqref{int_sim}. 
Moreover, from Lemma~\ref{prop}~(ii) we get $1\leq\alpha\leq 2$. \medskip

(ii) If $1<\alpha<2$, then $\Lambda(\{0\})=0$. 
Lemma \ref{deriv} provides that for $\al<2$ the function $H(u)$ is regularly varying with exponent $2-\al$ iff $h(u)$ is regularly varying with exponent $1-\al$  and then  
\[(2-\alpha)H(u)\ \sim\ uh(u)\]
as $u\to 0^+$. 
Applying Lemma \ref{deriv} once more  for $\al>1$, 
$h(u)$ is regularly varying with exponent $1-\al$ iff $\int_{(u,1]}\frac{\Lambda(dp)}{p^2}$ is regularly varying with exponent $-\al$ and then 
\[(\alpha-1)h(u)\ \sim\ u\int_{(u,1]}\frac{\Lambda(dp)}{p^2}\]
as $u\to 0^+$.
Bringing both asymptotics together with statement (i) finishes the proof. 
\end{proof}

\bigskip

\section{The length of a random external branch} \label{sec_rand}
Recall that $T^n$ denotes the length of an external branch picked at random.
The following result on its distribution function does not only play a decisive role in the proofs of Theorem~\ref{dustless} and~\ref{indep} but is also of interest on its own. It shows that the distribution of $T^n$ is primarily determined by the rate function $\mu$.  

\medskip

\newpage

\begin{theorem} \label{thm:int} 
For a $\Lambda$-coalescent without a dust component and a sequence $\left(r_n\right)_{n\in\N }$ satisfying $1< r_n\leq n$ for all $n\in\N $, 
we have 
\begin{align}\label{mu_int}
\PP\left(T^{n}> \int_{r_n}^n\frac{dx}{\mu(x)}\right) \eq\frac{\mu(r_n)}{\mu(n)}+\oo(1)
\end{align}
as $n\to\infty$. Moreover, 
\begin{align}\label{ineq}
\left(\frac{r_n}{n}\right)^2+\oo(1) \loe \PP\left(T^{n}> \int_{r_n}^n\frac{dx}{\mu(x)}\right) \loe \frac{r_n}{n}+\oo(1)
\end{align}
as $n\to\infty$. 
\end{theorem}

\medskip

Observe that the integral $\int_{r_n}^n\frac{dx}{\mu(x)}$ is the asymptotic time needed to go from $n$ to $r_n$ lineages according to Proposition \ref{ErgLem1}. 

\medskip 

For the proof, we recall our notations. $N_n=(N_n(t))_{t\ge 0}$ denotes the block counting process, with the embedded Markov chain $X=(X_j)_{j\in\N_0}$. 
In particular, we have $N_n(0)=X_0=n$ and we set $X_j=1$ for $j\geq\tau_n$, where $\tau_n$ is defined as the total number of merging events. 
The waiting time of the process $N_n$ in state $X_j$ is again referred to as $W_j$ for $0\leq j\leq \tau_n-1$.    
The number of merging events until the external branch ending in leaf $i\in\left\{1,\ldots,n\right\}$ coalesces is given by
\[\zeta_i^n\ :=\ \max{\left\{j\geq 0: \ \left\{i\right\}\in\Pi_n(W_0+\cdots+W_{j-1})\right\}}.\]  
Similarly, $\zeta^n$ denotes the corresponding number of a random external branch with length $T^n$. 

\medskip

\begin{proof}[Proof of Theorem \ref{thm:int}] 
For later purposes, we show the stronger statement
\begin{align}\label{cond}
\PP\left(T^{n} > \int_{r_n}^n\frac{dx}{\mu(x)}\,\bigg|\,N_n\right) \eq \frac{\mu(r_n)}{\mu(n)}+\oo_P(1)
\end{align} 
as $n\to\infty$. It implies \eqref{mu_int} by taking expectations and using dominated convergence. 
The statement \eqref{ineq} is a direct consequence in view of Lemma \ref{prop} (ii).  

\medskip
 
In order to prove \eqref{cond}, note that, by the standard subsubsequence argument and the metrizability of the convergence in probability, we can assume that $r_n/n$ converges to some value $q$ with $0\leq q\leq 1$. We distinguish three different cases of asymptotic behavior of the sequence $r_n/n$: 

\medskip

(a) We begin with the case $r_n\sim qn$ as $n\to\infty$, where $0<q<1$. Then there exist $q_1,q_2\in\left(0,1\right)$ such that $q_1n\leq r_n\leq q_2n$ for all $n\in\N $ but finitely many. \\ 
Let us first consider the discrete embedded setting and afterwards insert the time component. Since there are $\Delta X_{0}+1$ branches involved in the first merger, we have 
\[\PP\left(\zeta^n\ge 1\mid N_n\right) \eq 1-\frac{\Delta X_0+1}{X_0} \eq \frac{X_1-1}{X_0}\quad a.s.\]
Iterating this formula, it follows
\[\PP\left(\zeta^n\ge k\mid N_n\right) \eq \prod_{j=0}^{k-1}\frac{X_{j+1}-1}{X_j}\eq \frac{X_k-1}{n-1}\prod_{j=0}^{k-1}\left(1-\frac{1}{X_j}\right)\quad \text{a.s.}\]
for $k\geq 1$.  For a combinatorial treatment of this formula, see  \cite[Lemma 4]{DK18}.
Note that $\sum_{j=0}^{k-1}X_j^{-2}\leq\sum_{m=X_{k-1}}^\infty m^{-2}\leq 2\left(X_{k-1}\right)^{-1}$ 
to obtain via a Taylor expansion that
\begin{align}\label{Mac}
\PP\left(\zeta^n\geq k\left.\right|N_n\right)\eq\frac{X_k-1}{n-1}\exp{\Bigg(-\sum_{j=0}^{k-1}\frac{1}{X_j}+\OO\left(X_{k-1}^{-1}\right)\Bigg)} \qquad \text{a.s.}
\end{align}
as $n\rightarrow\infty$.

\medskip

We like to evaluate this quantity at the stopping times
\[\rho_{r_n}\ :=\ \min\{j\geq 0:\, X_j\leq r_n\}.\] 
From Lemma \ref{prop} (i) and (iii), we know that the function $\mu(x)$ is increasing in $x$ and that $x/\mu(x)$ converges in the dustless case to $0$ as $x\to\infty$. In view of $r_n\geq q_1n $, therefore, we have
\[\int_{r_n}^n\frac{dx}{\mu(x)}\loe \frac{n-r_n}{\mu(r_n)} \loe \left(\frac{1}{q_1}-1\right)\frac{r_n}{\mu(r_n)} \eq \oo(1).\] 
Hence, we may apply Proposition \ref{ErgLem2}  and obtain
\[\sum_{j=0}^{\rho_{r_n}-1}\frac{1}{X_j}\eq\log{\left(\frac{\mu(n)}{n}\frac{X_{\rho_{r_n}}}{\mu(X_{\rho_{r_n}})}\right)}+\oo_P(1).\]
Also, Lemma 3 of \cite{DK18} implies 
\[X_{\rho_{r_n}}
\eq r_n +\OO_P\left(\Delta X_{\rho_{r_n}}\right)\eq r_n+\oo_P(X_{\rho_{r_n}}).\]
Inserting these two estimates into equation \eqref{Mac} and using Lemma \ref{prop} (ii), it follows
\begin{align}\label{discrete}
\PP\left(\zeta^n\geq \rho_{r_n}\left.\right|N_n\right) \eq \frac{X_{\rho_{r_n}}-1}{n-1}\frac{\mu\left(X_{\rho_{r_n}}\right)}{X_{\rho_{r_n}}}\frac{n}{\mu\left(n\right)}
(1+\oo_P(1))\eq \frac{\mu\left(r_n\right)}{\mu\left(n\right)}+\oo_P(1).
\end{align}

In order to transfer this equality to the continuous-time setting, 
we first show that for each $\e\in(0,1)$ there is an $\delta>0$ such that 
\begin{align}\label{15}
(1+\delta)\int_{(1+\e)r_n}^n\frac{dx}{\mu(x)}\ <\ \int_{r_n}^n\frac{dx}{\mu(x)}\ <\ (1-\delta)\int_{(1-\e)r_n}^n\frac{dx}{\mu(x)}
\end{align}
for large $n\in\N $.
For the proof of the left-hand inequality, note that due to Lemma \ref{prop} (ii) we have
\[\frac{1}{n-(1+\e)r_n}\int_{(1+\e)r_n}^n\frac{dx}{\mu(x)}\ \leq\ \frac{1}{n-r_n}\int_{r_n}^n\frac{dx}{\mu(x)}\]
implying with $q_1n\leq r_n$ that 
\[\frac{1}{1-\e\frac{q_1}{1-q_1}}\int_{(1+\e)r_n}^n\frac{dx}{\mu(x)}\ \leq\ \frac{1}{1-\e\frac{r_n}{n-r_n}}\int_{(1+\e)r_n}^n\frac{dx}{\mu(x)}\ \leq\ \int_{r_n}^n\frac{dx}{\mu(x)}.\]
These inequalities show how to choose $\delta>0$. The right-hand inequality in \eqref{15} follows along the same lines.\\ 
Now, recalling the notion
\[\widetilde{\rho}_{r_n}\ :=\ \inf\{t\geq 0:\, N_n(t)\leq r_n\},\] 
Proposition \ref{ErgLem1} gives for sufficiently small $\e>0$ the formula
\begin{align}\label{rho_tilde}
\widetilde{\rho}_{r_n(1+\e)}\eq \int_{r_n(1+\e)}^n\frac{dx}{\mu\left(x\right)}\left(1+\oo_P(1)\right)
\end{align}
as $n\rightarrow\infty$.  
Combining \eqref{discrete} to \eqref{rho_tilde} yields
\begin{align*}
\PP\Big(T^n>&\int_{r_n}^n\frac{dx}{\mu(x)}\ \bigg|\, N_n\Big)\\[1ex]
& \loe \PP\bigg(T^n\geq\left(1+\delta\right)\int_{r_n(1+\e)}^n\frac{dx}{\mu(x)}\ \bigg|\, N_n\bigg)\\[1ex]
& \loe \PP\left(\left.T^n\geq\widetilde{\rho}_{r_n(1+\e)}\right|N_n\right)+\PP\bigg(\left(1+\delta\right)\int_{r_n(1+\e)}^n\frac{dx}{\mu(x)}<\widetilde{\rho}_{r_n(1+\e)}\ \bigg|\, N_n\bigg)\\[1ex]
& \eq \ \PP\left(\left.\zeta^n\geq \rho_{r_n(1+\e)}\right|N_n\right)+\oo_P(1)\\[1ex]
& \eq \ \frac{\mu(r_n(1+\e))}{\mu(n)}+\oo_P(1)\\[1ex]
& \loe \ \frac{\mu(r_n)}{\mu(n)}\left(1+\e\right)^2+\oo_P(1),
\end{align*}
where we used Lemma \ref{prop} (ii) for the last inequality. 
With this estimate holding for all $\e>0$, we end up with
\begin{align*} 
\PP\left(\left.T^n>\int_{r_n}^n\frac{dx}{\mu(x)}\right|N_n\right)\loe \frac{\mu(r_n)}{\mu(n)}+\oo_P(1)
\end{align*}
as $n\to\infty$. 
The reverse inequality can be shown in the same way so that we obtain equation \eqref{cond}. \bigskip

(b) Now we turn to the two remaining cases $r_n\sim n$ and $r_n=\oo(n)$. In view of Lemma \ref{prop} (ii), the asymptotics $r_n\sim n$ implies $\mu(r_n)\sim\mu(n)$, i.e., the right-hand side of \eqref{cond} converges to $1$. Furthermore, the sequence $(r_n')_{n\in\N }:=(qr_n)_{n\in\N }$, $0<q<1$, fulfills the requirements of part (a). With respect to Lemma \ref{prop} (ii), part (a), therefore, entails for all $q\in(0,1)$,
\begin{align*}
\PP\left(\left.T^n>\int_{r_n}^n\frac{dx}{\mu(x)}\right|N_n\right)\goe\PP\left(\left.T^n>\int_{r_n'}^n\frac{dx}{\mu(x)}\right|N_n\right)
 \goe \frac{\mu(qn)}{\mu(n)} +\oo_P(1) \goe q^2 +\oo_P(1) 
\end{align*} 
as $n\to\infty$. Hence, the left-hand side of \eqref{cond} also converges to $1$ in probability. Similarly, the convergence of both sides of \eqref{cond} to $0$ can be shown for $r_n=\oo(n)$.
\end{proof}

\bigskip

\section{Proofs of Theorem \texorpdfstring{\protect\ref{dustless} and \protect\ref{indep}}{1.1 and 1.2}}\label{sec_proofs}

\begin{proof}[Proof of Theorem \ref{dustless}]
Let $r_n$ be as required in Theorem \ref{thm:int}. Applying Lemma \ref{prop} (ii), we obtain
\[\int_{r_n}^n\frac{dx}{x} \loe \frac{\mu(n)}{n}\int_{r_n}^n\frac{dx}{\mu(x)} \loe  \int_{r_n}^n\frac{n-1}{x(x-1)}dx.\] 
Observing 
\[\int_{r_n}^n\frac{dx}{x} \eq \log{\frac{n}{r_n}}\]
and
\[\int_{r_n}^n\frac{n-1}{x(x-1)}dx \eq (n-1)\log{\frac{r_n-nr_n}{n-nr_n}},\]
Theorem \ref{thm:int} entails
\begin{align}\label{estim_down}
\PP\left(\frac{\mu(n)}{n}\,T^{n} >  \log{\frac{n}{r_n}}\right)\goe \left(\frac{r_n}{n}\right)^2+\oo(1)
\end{align}
and
\begin{align}\label{estim_up}
\PP\left(\frac{\mu(n)}{n}\,T^{n} > (n-1)\log{\frac{r_n-nr_n}{n-nr_n}}\right) \loe  \frac{r_n}{n}+\oo(1)
\end{align} 
\enlargethispage{2\baselineskip}
as $n\to\infty$, respectively. \\

Now let $t\geq 0$. Using equation \eqref{estim_down} for \[r_n \eq ne^{-t},\] while choosing 
\[r_n \eq \frac{ne^{t/(n-1)}}{1+n(e^{t/(n-1)}-1)}\]
in \eqref{estim_up}, we arrive at
\[e^{-2t}+\oo(1) \loe \PP\left(\frac{\mu(n)}{n}\,T^{n}> t\right) \loe \frac{e^{t/(n-1)}}{1+n(e^{t/(n-1)}-1)}+\oo(1) \eq \frac{1}{1+t}\left(1+\oo(1)\right)
,\]
as required. 
\end{proof}

\smallskip

\begin{proof}[Proof of Theorem \ref{indep}]
First, we treat the dustless case. Similar to the proof of Theorem \ref{thm:int}, we first consider the discrete version $\zeta_i^n$ of $T_i^n$ for $1\leq i\leq k$ to prove 
\begin{align}\label{discrete2}
\PP\left(\zeta_1^n\geq I_1^n,\,\ldots,\,\zeta_k^n\geq I_k^n\left.\right|N_n\right) \eq \PP\left(\zeta_1^n\geq I_1^n\left.\right|N_n\right)\,\cdots\,\PP\left(\zeta_k^n\geq I_k^n\left.\right|N_n\right)+\oo_P(1)
\end{align}
as $n\to\infty$, where $0=:I_0^n\leq I_1^n\leq\cdots\leq I_k^n$ are random variables measurable with respect to the $\sigma$-fields $\sigma\left(N_n\right)$. 
Denote by $\zeta_A$ the number of mergers until some external branch out of the set $A\subseteq\left\{1,\ldots,n\right\}$ coalesces and let $a:=\#A$.  
Given $\Delta X_j$, the $j$-th merging amounts to choosing $\Delta X_{j}+1$ branches uniformly at random out of the $X_j$ present ones implying 
\begin{align}\label{modul}
\PP\left(\zeta_A\geq m\left.\right|N_n\right)\eq\frac{(X_m-1)\cdots(X_m-a)}{(n-1)\cdots(n-a)}\prod_{j=0}^{m-1}\left(1-\frac{a}{X_j}\right) \qquad a.s.
\end{align}
for $m\geq 1$ (for details see (28) of \cite{DK18}).  
Let $\bar{\zeta}_{\left\{1,\ldots,k\right\}}:=\zeta_{\left\{1,\ldots,k\right\}}$ and $\bar{\zeta}_{\left\{i,\ldots,k\right\}}:=\zeta_{\left\{i,\ldots,k\right\}}-\zeta_{\left\{i-1,\ldots,k\right\}}$ for $2\leq i\leq k$. Moreover, let $\widebar{N}_{X_j}(t):=N_n(t+W_0+\cdots+W_{j-1})$, in particular, $\widebar{N}_{X_0}(t):=N_n(t)$. The Markov property and \eqref{modul} provide  
\enlargethispage{2\baselineskip}
\begin{align*}
\PP\Big(\zeta_1^n\geq I_1^n,&\ldots,\zeta_k^n\geq I_k^n\Big|N_n\Big)\\[1.5ex]
& \eq  \prod_{i=1}^{k}\PP\Big(\bar{\zeta}_{\left\{i,\ldots,k\right\}}\geq I_i^n-I_{i-1}^n\Big|\widebar{N}_{X_{I_{i-1}^n}}\Big)\\[1.5ex]
& \eq  \prod_{i=1}^{k}\left[\frac{(X_{I_i^n}-1)\cdots(X_{I_i^n}-k+i-1)}{(X_{I_{i-1}^n}-1)\cdots(X_{I_{i-1}^n}-k+i-1)}\prod_{j=I_{i-1}^n}^{I_i^n-1}\left(1-\frac{k-i+1}{X_j}\right)\right]\\[1.5ex]
& \eq \prod_{i=1}^{k}\left[\frac{(X_{I_i^n}-k+i-1)}{(n-k+i-1)}\prod_{j=I_{i-1}^n}^{I_i^n-1}\left(1-\frac{k-i+1}{X_j}\right)\right] \qquad a.s.
\end{align*} 
For $1\leq i\leq k$, note that
\[\left(1-\frac{k-i+1}{X_j}\right)\eq\left(1-\frac{1}{X_j}\right)^{k-i+1}+\OO\left(X_j^{-1}\right)\]
and
\[\frac{X_{I_i^n}-k+i-1}{n-k+i-1}\eq\frac{X_{I_i^n}-1}{n-1}+\OO\left(n^{-1}\right)\]
to obtain
\begin{align*}
\PP\big(\zeta_1^n\geq I_1^n,&\ldots,\zeta_k^n\geq I_k^n\left.\right|N_n\big)\\[1.5ex]
 & \eq  \prod_{i=1}^{k}\left[\left(\frac{X_{I_i^n}-1}{n-1}+\OO\left(n^{-1}\right)\right)\left(\,\prod_{j=I_{i-1}^n}^{I_i^n-1}\left(1-\frac{1}{X_j}\right)^{k-i+1}+\OO\left(\left(X_{I_i^n}-1\right)^{-1}\right)\right)\right]\\[1.5ex]
 & \eq  \prod_{i=1}^{k}\left[\frac{X_{I_i^n}-1}{n-1}\,\prod_{j=I_{i-1}^n}^{I_i^n-1}\left(1-\frac{1}{X_j}\right)^{k-i+1}\right]+\oo_P(1)\\[1.5ex]
&\eq \prod_{i=1}^{k}\left[\frac{X_{I_i^n}-1}{n-1}\prod_{j=0}^{I_i^n-1}\left(1-\frac{1}{X_j}\right)\right]+\oo_P(1)
\end{align*}
as $n\to\infty$, where the rightmost $\OO(\cdot)$-term in the first line stems from the fact that $X_{I_i^n}<X_j$ for all $ j< I_i^n$.  
Furthermore, from \eqref{modul}
with $A=\{i\}$, we know that 
\[\PP\left(\zeta_i^n\geq I_i^n\left|\right.N_n\right)\eq\frac{X_{I_i^n}-1}{n-1}\prod_{j=0}^{I_i^n-1}\left(1-\frac{1}{X_j}\right) \qquad a.s.\]
so that we arrive at equation \eqref{discrete2}.
\\ 
Now based on exchangeability, it is no loss to assume that $0\leq t_1^n\leq\cdots\leq t_k^n$. So 
inserting 
\[I_i^n\ :=\ \min\bigg\{k\geq 1:\ \sum_{j=0}^{k-1} W_j > t_i^n\bigg\}\wedge\tau_n\]
in \eqref{discrete2} yields 
\enlargethispage{2\baselineskip} 
\begin{align*}
\PP\left(\left.T_1^n\,>\, t_1^n,\,\ldots,\,T_k^n\,>\, t_k^n\right|N_n\right) & \eq \PP\left(\zeta_1^n\geq I_1^n,\,\ldots,\,\zeta_k^n\geq I_k^n\left.\right|N_n\right)\\[.5ex]
& \eq \prod_{i=1}^k \PP\left(\zeta_i^n\geq I_i^n\left.\right|N_n\right) +\oo_P(1)\\ 
&\eq \prod_{i=1}^k \PP\left(\left.T_i^n\,>\, t_i^n\right|N_n\right)
+\oo_P(1)  
\end{align*}
as $n\to\infty$.  
For $1\leq i\leq k$, let $1<r_i^n\leq n$ be defined implicitly via 
\[t_i^n\eq \int_{r_i^n}^n\frac{dx}{\mu(x)}\,.\] 
From Lemma \ref{prop} (iii) we know that $\int_1^n\frac{dx}{\mu(x)}=\infty\,$; therefore, $r_i^n$ is well-defined. 
In the dustless case, consequently, we may apply formula \eqref{cond} to obtain
\begin{align*}
\PP\left(\left.T_1^n\,>\, t_1^n,\,\ldots,\,T_k^n\,>\, t_k^n\right|N_n\right) & \eq \prod_{i=1}^k \PP\left(\left.T_i^n\,>\, t_i^n\right|N_n\right) +\oo_P(1)
\\[.5ex]
& \eq 
\prod_{i=1}^k\frac{\mu(r_i^n)}{\mu(n)}+\oo_P(1) 
\end{align*}
as $n\to\infty$. Taking expectations in this equation yields, via dominated convergence, the theorem's claim for $\Lambda$-coalescents without a dust component.

\bigskip

For $\Lambda$-coalescents with dust, we use for $t>0$ the formula
\[\lim_{n\to\infty}\PP\left(T_1^n>t,\ldots,T_k^n>t\right) \eq \E\left[S_t^k\right],\]
with non-degenerative positive random variables $S_t$ (see (10) in \cite{Moe10}).  For $k\geq 2$, Jensen's inequality implies
\[\lim_{n\to\infty}\PP\left(T_1^n>t,\ldots,T_k^n>t\right) > \E\left[S_t\right]^k \eq \lim_{n\to\infty}\PP\left(T_1^n>t,\ldots,T_k^n>t\right).\]
This finishes the proof. 
 
\end{proof}

\bigskip

\section{Proof of Theorem \texorpdfstring{\protect\ref{iff}}{1.3}}\label{sec_proof}

(a) First suppose that $\mu(x)$ is regularly varying with exponent $\alpha\in[1,2]$, i.e., we have
\begin{align}\label{varying}
\mu(x)\eq x^\alpha L(x),
\end{align}
where $L$ is a slowly varying function. Let $r_n:=qn$ with $0<q\leq 1$.  
The statement of Theorem~\ref{thm:int} then boils down to
\begin{align}\label{continuous2}
\PP\left(\frac{\mu(n)}{n}T^{n}\,>\, \frac{1}{n} \int_{qn}^n\frac{\mu(n)}{\mu(x)}dx\right)  \eq q^\alpha+\oo(1)
\end{align}
as $n\rightarrow\infty$.
From \eqref{varying} we obtain 
\[n^{-1}\int_{qn}^n\frac{\mu\left(n\right)}{\mu\left(x\right)}dx  
\ \sim\  
\begin{cases}
\; -\log{q} \quad &\text{ for } \quad \alpha =1\\[1ex]
\; \frac{1}{\alpha-1}\left(q^{-(\alpha-1)}-1\right) \quad &\text{ for } \quad 1<\alpha\leq 2
\end{cases}\]
as $n\rightarrow\infty$. 
Thus, choosing, for given $t\geq 0$, 
\[q \eq
\begin{cases}
\; e^{-t} \quad &\text{ for } \quad \alpha=1\\[1ex]
\; \left(1+\left(\alpha-1\right)t\right)^{-\frac{1}{\alpha-1}} \quad &\text{ for } \quad 1<\alpha\leq 2
\end{cases}\]
in equation (\ref{continuous2}) yields the claim. 
\medskip

(b) Now suppose that $\gamma_n\,T^n$ converges for some positive sequence $(\gamma_n)_{n\in\N }$ in distribution as $n\rightarrow\infty$ to a probability measure unequal to $\delta_0$ with cumulative distribution function $F=1-\widebar{F}$, i.e.,
\begin{align}\label{=>}
\PP\left(\gamma_n\, T^n> t\right)\ \stackrel{n\to\infty}{\longrightarrow} \ \widebar{F}(t)
\end{align}
for $t\geq 0$, $t\notin D$, where $D$ denotes the set of discontinuities of $\widebar{F}$. Note that $0<\widebar{F}(t)<1$ for all $t>0$ due to Theorem \ref{dustless}. 
In order to prove that $\mu$ is regularly varying, we bring together the assumption \eqref{=>} with the statement of Theorem \ref{thm:int}, which requires several steps.

For this purpose we define, similarly as in the proof of Theorem \ref{indep}, the numbers $r_n(t)$ for $t\geq 0$ implicitly via
\begin{align}\label{chi}
t\eq \gamma_n\int_{r_n(t)}^n\frac{dx}{\mu(x)}\,.
\end{align}
Let us first solve this implicit equation. 
Applying formula \eqref{cond} and \eqref{=>}, we obtain 
\begin{align}\label{sim}
\frac{\mu(r_n(t))}{\mu(n)} 
\eq \widebar{F}(t)+\oo(1)
\end{align}  
for all $t\geq 0$, $t\notin D$, as $n\to\infty$. Differentiating both sides of \eqref{chi} with respect to $t$ and using Lemma \ref{prop} (i) yields 
\[\left|\frac{\gamma_n\,r'_n(t)}{\mu(n)}\right|\eq\frac{\mu(r_n(t))}{\mu(n)} \loe 1.\]
In conjunction with \eqref{sim}, it follows that
\[\frac{\gamma_n\,r'_n(t) }{\mu(n)}\eq-\widebar{F}(t)+\oo(1)\]
and, by dominated convergence, 
\begin{align}\label{r_n}
r_n(t)\eq n-\frac{\mu(n)}{\gamma_n}\left(\int_0^t\widebar{F}(s)ds +\oo\left(1\right)\right) 
\end{align} 
as $n\to\infty$. 

\medskip

Next, we show that
\(\gamma_n \sim \ c\mu'(n) \)
for some $c>0$. 
From Theorem \ref{dustless} it follows that there exist $0<c_1\leq c_2<\infty$ with 
\begin{align}\label{mu_eta}
c_1\,\frac{\mu(n)}{n}\loe\gamma_n\loe c_2\,\frac{\mu(n)}{n}, \qquad n\geq 2.
\end{align}

Furthermore, from equation \eqref{r_n} and a Taylor expansion, we get
\[\mu(r_n(t))\eq\mu(n)+\mu'(n)\left(r_n(t)-n\right)+\frac{1}{2}\mu''\left(\xi_n\right)\left(r_n(t)-n\right)^2,\]
where $r_n(t)\leq\xi_n\leq n$. Dividing this equation by $\mu(n)$, using \eqref{sim} and \eqref{r_n}, as well as rearranging terms, we obtain
\[\bigg|1-\widebar{F}(t)+\oo(1)-\frac{\mu'(n)}{\gamma_n}\int_0^t\widebar{F}(s)ds\left(1+\oo(1)\right)\bigg| 
\eq \frac{\mu''(\xi_n)\mu(n)}{2\gamma_n^2}\left(\int_0^t\widebar{F}(s)ds\right)^2\left(1+\oo(1)\right)\]
as $n\to\infty$. 
From Lemma \ref{prop} (iii) and (i), we get $\mu''(\xi_n)\leq\mu'(\xi_n)/(\xi_n-1)\leq \mu'(n)/(r_n(t)-1)$. 
Moreover, equation \eqref{r_n} with \eqref{mu_eta} yields $r_n(t)-1\geq n/2+\oo(n)$ for $t$ sufficiently small.  
Taking \eqref{mu_eta} once more into account, we obtain that for given $\e>0$ and $t$ sufficiently small, 
\begin{align*}
\bigg|1-\widebar{F}(t)+\oo(1)-\frac{\mu'(n)}{\gamma_n}\int_0^t\widebar{F}(s)ds\left(1+\oo(1)\right)\bigg| &\loe \frac{\mu'(n)}{c_1\gamma_n}\left(\int_0^t\widebar{F}(s)ds\right)^2\left(1+\oo(1)\right)\\[1ex]
& \loe \e\,\frac{\mu'(n)}{\gamma_n}\left(\int_0^t\widebar{F}(s)ds\right)\left(1+\oo(1)\right)
\end{align*}
or equivalently, for $t>0$,
\[\bigg|\frac{\gamma_n}{\mu'(n)}\ -\ \frac{\int_0^t\widebar{F}(s)ds}{1-\widebar{F}(t)}\left(1+\oo(1)\right)\bigg|\loe \e\,\frac{\int_0^t\widebar{F}(s)ds}{1-\widebar{F}(t)}\left(1+\oo(1)\right).\]
The right-hand quotient is finite and positive for all $t>0$, which implies our claim $\gamma_n\sim c\mu'(n)$ for some $c>0$. \bigskip

We now remove $\gamma_n$ from our equations by setting $\gamma_n=\mu'(n)$, without loss of generality. With this choice \eqref{mu_eta} changes into
\begin{align*}
c_1\,\frac{\mu(n)}{n}\loe \mu'(n) \loe c_2\,\frac{\mu(n)}{n}, \qquad n\geq 2.
\end{align*} 
Also, inserting \eqref{r_n} and \eqref{mu_eta} in \eqref{sim} yields
\begin{align*}
\mu(n)\widebar{F}(t)\left(1+\oo(1)\right)\eq\mu\left(r_n(t)\right)\eq\mu\left(n-\frac{\mu(n)}{\mu'(n)}\int_0^t\widebar{F}(s)ds+\oo(n)\right)
\end{align*}
as $n\to\infty$.
Let us suitably remodel these formulae. In view of the monotonicity properties of $\mu$ and $\mu'$ due to Lemma \ref{prop} (i), we may proceed to
\begin{align} \label{mu_eta1}c_3\frac{\mu(x)}{x}\loe\mu'(x)\loe c_4\frac{\mu(x)}{x},\quad x\geq 2,
\end{align}
for suitable $0< c_3\leq c_4<\infty$, as well as
\begin{align}\nonumber
\mu(x)\widebar{F}(t)& \eq \mu\left(x-\frac{\mu(x)}{\mu'(x)}\int_0^t\widebar{F}(s)ds+\oo(x)\right)\left(1+\oo(1)\right) \\[1ex]
&\eq \mu\left(x-\frac{\mu(x)}{\mu'(x)}\int_0^t\widebar{F}(s)ds+\oo(x)\right) \label{mu_phi}
\end{align}
as $x\to\infty$, where we pushed the $(1+\oo(1))$-term into $\mu$ by means of Lemma \ref{prop} (ii).  
This equation 
suggests to pass to the inverse of $\mu$.     
From Lemma~\ref{prop}~(i) we know that $\mu(x)$ has an inverse $\nu(y)$.  
For this function,  formula \eqref{mu_eta1} translates into
\begin{align}\label{phi_d}
\frac{\nu(y)}{c_4y}\loe\nu'(y)\loe\frac{\nu(y)}{c_3y}.
\end{align}
Also, applying $\nu$ to equation \eqref{mu_phi}, both inside and outside, we get  
\[\nu \left(y\widebar{F}(t)\right)\eq\nu(y)-y\,\nu'(y)\int_0^t\widebar{F}(s)ds+\oo(\nu(y)).\]

\medskip

This equation allows us, in a next step, to further analyse $\widebar{F}$.  
With $0\leq u<v$, $u,v\notin D$, it follows that
\begin{align}\label{diff_phi}
\nu\left(\widebar{F}(u)y\right)-\nu\left(\widebar{F}(v)y\right)\eq y\,\nu'(y)\int_u^v\widebar{F}(s)ds \left(1+\oo(1)\right)
\end{align}
as $y\to\infty$. This equation immediately implies that $\widebar{F}(v)<\widebar{F}(u)$ for all $u<v$. It also shows that $\widebar{F}$ has no jump discontinuities, i.e., $D=\emptyset$.
Indeed, by the mean value theorem and because $\nu'(y)=1/\mu'(\nu(y))$ is decreasing due to Lemma \ref{prop} (i), we have for $0\leq u<v$,  
\[\nu\left(\widebar{F}(u)y\right)-\nu\left(\widebar{F}(v)y\right)\goe\nu'(y\widebar{F}(u))y\left(\widebar{F}(u)-\widebar{F}(v)\right)\goe\nu'(y)y\left(\widebar{F}(u)-\widebar{F}(v)\right).\]
Thus, also assuming $u,v\notin D$, \eqref{diff_phi} yields
\[\widebar{F}(u)-\widebar{F}(v)\loe\int_u^v\widebar{F}(s)ds\loe v-u,\]
which implies $D=\emptyset$. 

\medskip

Now, we are ready to show that $\nu$ and, therefore, $\mu$ is regularly varying. By a Taylor expansion, we get
\[\nu(\widebar{F}(v)y)-\nu(\widebar{F}(u)y)\eq-\nu'(\widebar{F}(u)y)y(\widebar{F}(u)-\widebar{F}(v))+\frac{1}{2}\nu''(\xi_y)y^2(\widebar{F}(u)-\widebar{F}(v))^2,\]
where $\widebar{F}(v)y\leq\xi_y\leq\widebar{F}(u)y$. 
Dividing this equation by $y\nu'(y)$, using formula \eqref{diff_phi} and rearranging terms, it follows that for $y\to\infty$,  

\begin{align}
\label{new^2}\left|\int_u^v\widebar{F}(s)ds(1+\oo(1))\ -\ \frac{\nu'(\widebar{F}(u)y)}{\nu'(y)}\left(\widebar{F}(u)-\widebar{F}(v)\right)\right|
\eq\frac{1}{2}\frac{\nu''(\xi_y)y}{\nu'(y)}(\widebar{F}(u)-\widebar{F}(v))^2.
\end{align}
Next, let us bound the right-hand term. Note that from Lemma \ref{prop} (iii) we have, for $y$ sufficiently large, 
\[\left|\nu''(y)\right|\eq\nu'(y)^2\frac{\mu''(\nu(y))}{\mu'(\nu(y))}\loe\frac{\nu'(y)^2}{\nu(y)-1} \loe \frac{2\nu'(y)^2}{\nu(y)}  \,.\]

Hence, using \eqref{phi_d} twice and $\widebar{F}(v)y\leq\xi_y\leq\widebar{F}(u)y$, it follows, for $y$ sufficiently large, 
\[\frac{1}{2}\nu''(\xi_y) \loe \frac{\nu'(\xi_y)^2}{\nu(\xi_y)}  \loe \frac{1}{c^2_3} \frac{\nu(\xi_y)}{\xi_y^2}\loe \frac{\nu(\widebar{F}(u)y)}{\widebar{F}(v)^2y^2} \loe \frac{c_4}{c_3^2}\frac{\nu'(\widebar{F}(u)y)\widebar{F}(u)}{\widebar{F}(v)^2y}.\]
Now, for given $u>0$ and given $\e>0$, because of the continuity and strict monotonicity of $\widebar{F}$, we get 

\[\frac{1}{2}\nu''(\xi_y) \loe \e \frac{\nu'(\widebar{F}(u)y)}{y(\widebar{F}(u)-\widebar{F}(v))}\]
if only the (positive) difference $v-u$ is sufficiently small.  Inserting into \eqref{new^2}, we get

\[\left|\int_u^v\widebar{F}(s)ds(1+\oo(1))\ -\ \frac{\nu'(\widebar{F}(u)y)}{\nu'(y)}\left(\widebar{F}(u)-\widebar{F}(v)\right)\right| \loe  \e\,\frac{\nu'(\widebar{F}(u)y)}{\nu'(y)}(\widebar{F}(u)-\widebar{F}(v))\]

or equivalently, for $y\to\infty$,
\[\Big|\frac{\nu'(y)}{\nu'(\widebar{F}(u)y)}\ -\ \frac{\widebar{F}(u)-\widebar{F}(v)}{\int_u^v\widebar{F}(s)ds}(1+\oo(1))\Big| \loe \e\,\frac{\widebar{F}(u)-\widebar{F}(v)}{\int_u^v\widebar{F}(s)ds}(1+\oo(1)).\]
Again, since the right-hand quotient is finite and positive for all $u<v$,
this estimate implies that $\nu'(y)/\nu'(\widebar{F}(u)y)$ has a positive finite limit as $y\to\infty$. Because $\widebar{F}(u)$ takes all values between $0$ and $1$, $\nu'(y)$ is regularly varying.  
From the Lemma in Section VIII.9 of \cite{Fell71}, we then obtain the regular variation of $\nu$ with some exponent $\eta\geq 0$. 
It fulfills $\frac{1}{2}\leq\eta\leq 1$ as Lemma \ref{prop} (ii) yields 
\[a\sqrt{y}\loe \nu(y)\loe b y\]
for some $a,\,b>0$.  
Hence, $\mu$, as the inverse function of $\nu$, is regularly varying with exponent $\alpha\in\left[1,2\right]$ (see Theorem 1.5.12 of \cite{BGT87}).
\qed

\bigskip

\section{Moment calculations for external branches \texorpdfstring{of $\mathbf{\Lambda}$-coalescents}{}}\label{sec_mom}

In this section, we consider the number of external branches $Y_j$ after $j$ merging events:
\[Y_j \ := \ \# \left\{1\leq i\leq n:  \ \{i\}\in\Pi_n\left(W_0+\cdots+W_{j-1}\right)\right\}.\]
In particular, we set $Y_0=n$ and $Y_j=0$ for $j>\tau_n$. (Again, we suppress $n$ in the notation, for convenience.) We provide a representation of the conditional moments of the number of external branches for general $\Lambda$-coalescents (also covering coalescents with a dust component).  
For this purpose, we use the notation  $\left(x\right)_r:=x\left(x-1\right)\cdots\left(x-r+1\right)$ for falling factorials with $x\in\R $ and $r\in\N $. 
Recall that $\tau_n$ is the total number of merging events. 

\medskip

\begin{lem} \label{Lambda}
Consider a general $\Lambda$-coalescent and let $\rho$ be a $\sigma(N_n)$-measurable random variable with $0\leq\rho\leq\tau_n$ a.s.  

\begin{enumerate}
	\item For a natural number $r$, the $r$-th factorial moment, given $N_n$, can be expressed as 
	\[\E\left[\left(Y_\rho\right)_r\left.\right|N_n\right] \eq \left(X_\rho\right)_r\,\prod_{j=1}^\rho \left(1-\frac{r}{X_j}\right) \eq \left(X_\rho-1\right)_r\,\frac{n}{n-r}\,\prod_{j=0}^{\rho-1} \left(1-\frac{r}{X_j}\right)\quad a.s.\]

	\item For the conditional variance, the following inequality holds: 
	\[\V\left(Y_\rho\left|N_n\right.\right) \loe  \E\left[Y_\rho\left|N_n\right.\right] \qquad a.s.\]
\end{enumerate}
\end{lem}

\begin{proof}
(i) First, we recall a link between the external branches and the hypergeometric distribution based on the Markov property and exchangeability properties of the $\Lambda$-coalescent, as already described for Beta-coalescents in \cite{DKW14}:\\ 
Given $N_n$ and $Y_0,\ldots,Y_{\rho-1}$, the $\Delta X_\rho+1$ lineages coalescing at the $\rho$-th merging event are chosen uniformly at random among the $X_{\rho-1}$ present ones. For the external branches, this means that, given $N_n$ and $Y_0,\ldots,Y_{\rho-1}$, the decrement $\Delta Y_\rho :=Y_{\rho-1}-Y_\rho$ has a hypergeometric distribution with parameters $X_{\rho-1},\, Y_{\rho-1}$ and $\Delta X_{\rho}+1$. 
In view of the formula of the $i$-th factorial moment of a hypergeometric distributed random variable, we obtain
\begin{align}\label{factmom}
\E\left[\left(\Delta Y_\rho\right)_i\left.\right|N_n,Y_0,\ldots,Y_{k-1}\right]\ =\ \left(\Delta X_\rho  +1\right)_i\frac{\left(Y_{\rho-1}\right)_i}{\left(X_{\rho-1}\right)_i} \qquad a.s.
\end{align}

Next, we look closer at the falling factorials. We have the following binomial identity 
\begin{align}\label{fallingfactorial}
\left(a-b\right)_r \eq (a)_r\sum_{i=0}^r\binom{r}{i}\left(-1\right)^i\frac{\left(b\right)_i}{(a)_i}
\end{align}
for $a,b\in\R $ and $r\in\N $. It follows from the Chu–Vandermonde identity (formula 1.5.7 in \cite{BM04})
\[(x+y)_r=\sum_{i=0}^r\binom{r}{i}(x)_i(y)_{r-i}\] 
with $x,y\in\R$ and the calculation
\begin{align*}
(a-b)_r&\eq(-1)^r(b+r-1-a)_r\\[1ex]
&\eq (-1)^r\sum_{i=0}^r\binom{r}{i}\left(b\right)_i(r-1-a)_{r-i}\\[1ex]
&\eq (-1)^r\sum_{i=0}^r\binom{r}{i}\left(b\right)_i(-1)^{r-i}\frac{(a)_r}{(a)_i}\,.
\end{align*} 
Returning to the number of external branches, we obtain from the identity \eqref{fallingfactorial} that
\[\left(Y_\rho\right)_r \eq (Y_{\rho-1})_r\sum_{i=0}^{r}\binom{r}{i}\left(-1\right)^{i}\frac{\left(\Delta Y_\rho\right)_i}{\left(Y_{\rho-1}\right)_{i}}\;.\] 
With equation \eqref{factmom}, we arrive at
\begin{align*}
\E\left[\left(Y_\rho\right)_r\left.\right|N_n,Y_0,\ldots,Y_{\rho-1}\right] \eq \left(Y_{\rho-1}\right)_r\sum_{i=0}^{r}\binom{r}{i}\left(-1\right)^i\frac{\left(\Delta X_\rho+1\right)_i}{\left(X_{\rho-1}\right)_i} \qquad a.s.
\end{align*}
Furthermore, combining the binomial identity \eqref{fallingfactorial} with the definition of $\Delta X_\rho$, we have
\[\left(X_\rho-1\right)_r \eq (X_{\rho-1})_r\sum_{i=0}^{r}\binom{r}{i}\left(-1\right)^i\frac{\left(\Delta X_\rho+1\right)_i}{\left(X_{\rho-1}\right)_{i}}\;.\]
Thus, 
\[\E\left[\left(Y_\rho\right)_r\left.\right|N_n,Y_0,\ldots,Y_{\rho-1}\right] \eq \left(Y_{\rho-1}\right)_r\frac{\left(X_\rho-1\right)_r}{\left(X_{\rho-1}\right)_r} \qquad a.s.\]
and, finally,
\[\frac{\E\left[\left(Y_\rho\right)_r\left.\right|N_n\right]}{\left(X_\rho\right)_r} \eq \frac{\E\left[\left(Y_{\rho-1}\right)_r\left.\right|N_n\right]}{\left(X_{\rho-1}\right)_r}\frac{\left(X_\rho-1\right)_r}{\left(X_\rho\right)_r} \eq \frac{\E\left[\left(Y_{\rho-1}\right)_r\left.\right|N_n\right]}{\left(X_{\rho-1}\right)_r}\left(1-\frac{r}{X_\rho}\right) \quad a.s.\]
The proof now finishes by iteration and taking $\E\left[Y_0\left|N_n\right.\right]=Y_0=X_0$ into account. 

\bigskip

(ii) The inequality for the conditional variance follows from the representation in (i) with $r=1$ and $r=2$:
\begin{align*}
\V\left(Y_\rho  \left|N_n\right.\right)& \eq X_\rho  \left(X_\rho -1\right)\prod_{j=1}^\rho \left(1-\frac{2}{X_j}\right)-X^2_\rho \prod_{j=1}^\rho \left(1-\frac{1}{X_j}\right)^2+X_\rho \prod_{j=1}^\rho \left(1-\frac{1}{X_j}\right)\\
& \loe X^2_\rho \prod_{j=1}^\rho \left(1-\frac{2}{X_j}\right)-X^2_\rho \prod_{j=1}^\rho \left(1-\frac{1}{X_j}\right)^2+X_\rho \prod_{j=1}^\rho \left(1-\frac{1}{X_j}\right)\\
& \loe  X_\rho \prod_{j=1}^\rho \left(1-\frac{1}{X_j}\right)\eq \E\left[Y_\rho\left|N_n\right.\right] \qquad a.s.
\end{align*}
This finishes the proof. 
\end{proof}

\bigskip

\section{Proof of Theorem \texorpdfstring{\protect\ref{reg_var}}{1.5}} \label{sec_proof_beta}

In order to study $\Lambda$-coalescents having a regularly varying rate of decrease $\mu$ with exponent $\alpha\in(1,2]$, we define
\[\kappa(x) \ := \ \frac{\mu(x)}{x}, \qquad x \geq 1,\]
for convenience. 
For $k\in\N $ and for real-valued random variables $Z_1,\ldots,Z_k$, denote the reversed order statistics by 
\[Z_{\left\langle 1\right\rangle} \ \geq\ \cdots \ \geq\ Z_{\left\langle k\right\rangle}.\]

\medskip
 
We now prove the following theorem that is equivalent to Theorem \ref{reg_var}. Recall the definition of $s_n$ in (\ref{new2}). 

\begin{theorem}\label{reg_var v2}
Suppose that the $\Lambda$-coalescent has a regularly varying rate $\mu$ with exponent $1<\alpha\leq 2$ and fix $\ell\in\N $. Then, as $n\to\infty$, the following convergence holds:
\[\kappa(s_n)\left(T_{\left\langle 1\right\rangle}^n,\ldots,T_{\left\langle \ell\right\rangle}^n\right)\ \stackrel{d}{\longrightarrow}\ \left(U_1,\ldots,U_\ell\right),\]
where $U_1>\cdots>U_\ell$ are the points in decreasing order of a Poisson point process $\Phi$ on $(0,\infty)$ with intensity measure $\phi(dx) \eq \alpha \left(\left(\alpha-1\right)x\right)^{-1-\al/(\al-1)}\;dx$. 
\end{theorem}

\medskip

For the rest of this section, keep the stopping times 
\begin{align}\label{stopp1}
\widetilde{\rho}_{c,n}\ :=\ \inf\left\{t\geq 0:\, N_n(t)\leq c s_n\right\}
\end{align}
in mind and define their discrete equivalents 
\begin{align}\label{stopp2}
\rho_{c,n}\ :=\ \min\left\{j\geq 0:\, X_j\leq c s_n\right\}
\end{align}
for $c>0$. 
Later, we shall apply Proposition \ref{ErgLem2} to the latter stopping times, in view of \eqref{s_n} and
\begin{align} \label{tilde_0}
\int_{cs_n}^n\frac{dx}{\mu(x)} \eq \OO\left(\int_{ cs_n}^n x^{-\alpha+\e}dx\right) \eq \OO\left(s_n^{1-\al+\e}\right) \eq \oo(1)
\end{align}
for $0<\e<\al-1$ because $\mu$ is regularly varying with exponent $\al$.\\

The next proposition deals with properties of the stopping times from \eqref{stopp1} and \eqref{stopp2}. It justifies the choice of $s_n$, it shows that $X_{\rho_{c,n}}$ diverges at the same rate as $s_n$ and that 
$Y_{\rho_{c,n}}$ is uniformly bounded in $n$. In particular, it reveals that for large $c$ there are with high probability external branches still present up to the times $\widetilde{\rho}_{c,n}$.   

\medskip

\begin{prop}\label{EV}
Assume that the $\Lambda$-coalescent has a regularly varying rate $\mu$ with exponent $\al\in(1,2]$. Then we have:  
\begin{enumerate}
\item For each $\e>0$, there exists $c_\e>0$ such that for all $c\geq c_\e$,  
\[\lim_{n\to\infty}\PP\left(\kappa(s_n)\,\widetilde{\rho}_{c,n}\geq\e\right) \eq 0.\]
\item For each $c>0$, as $n\to\infty$, 
\[X_{\rho_{c,n}} \eq cs_n +\oo_P(s_n).\]  
\item For each $\e>0$,
\[\limsup_{n\to\infty}\PP\left(\left|c^{-\al}\,Y_{\rho_{c,n}}-1\right|\geq\e\right)\ \stackrel{c\to\infty}{\longrightarrow}\ 0.\]
\end{enumerate}
\end{prop}

\begin{proof}
(i) Because $\mu$ is regularly varying with exponent $\al>1$, we have
\[\int_{cs_n}^\infty\frac{dx}{\mu(x)} \ \sim\ \frac{1}{\al-1}\frac{cs_{n}}{\mu(cs_n)}\ \sim \ \frac{1}{\al-1}c^{1-\al}\frac{1}{\kappa(s_n)}\]
as $n\to\infty$. 
Now Proposition \ref{ErgLem2} implies that
\[\kappa(s_n)\,\widetilde{\rho}_{c,n} \loe \frac{1}{\al-1}c^{1-\al}(1+\oo_P(1)),\]
which entails the claim.  \medskip

(ii) Because of \eqref{tilde_0}, we may use Lemma 3 (ii) of \cite{DK18}. In conjunction with the definition of $\rho_{c,n}$, therefore, we obtain
\begin{align*}
\frac{X_{\rho_{c,n}}}{X_{\rho_{c,n}-1}} \eq 1 - \frac{\Delta X_{\rho_{c,n}}}{X_{\rho_{c,n}-1}} \eq 1+\oo_P(1)
\end{align*}
as $n\to\infty$. 
This implies the statement because of $X_{\rho_{c,n}}\leq cs_n<X_{\rho_{c,n}-1}$.  \medskip

(iii) We first prove that
\begin{align}\label{condExp}
\E\left[Y_{\rho_{c,n}}|N_n\right] \eq c^\alpha + \oo_P\left(1\right)
\end{align}
as $n\to\infty$. 
Lemma \ref{Lambda} (i), together with a Taylor expansion as in \eqref{Mac}, provides 
\begin{align*}
\E\left[Y_{\rho_{c,n}}\left.\right|N_n\right] & = (X_{\rho_{c,n}}-1) \exp{\left(-\sum_{j=0}^{\rho_{c,n}-1}\frac{1}{X_j}+\OO\left(X_{\rho_{c,n}-1}^{-1}\right)\right)} 
\end{align*}
as $n\rightarrow\infty$. Furthermore, \eqref{s_n} and \eqref{tilde_0} allow us to apply Proposition \ref{ErgLem2} yielding    
\begin{align}\label{Prop.3}
\sum_{j=0}^{\rho_{c,n}-1}\frac{1}{X_j} \eq \log{\left(\frac{\kappa(n)}{\kappa(X_{\rho_{c,n}})}\right)}+\oo_P(1)
\end{align}
as $n\to\infty$. 
Combining statement (ii) with Lemma \ref{prop} (ii), therefore, we arrive at 
\[\E\left[Y_{\rho_{c,n}}\left.\right|N_n\right] \eq n\,\frac{\mu\big(X_{\rho_{c,n}}\big)}{\mu(n)}\left(1+\oo_P\left(1\right)\right)
\eq n\,\frac{\mu(cs_n)}{\mu(n)}\left(1+\oo_P\left(1\right)\right)
\]
 \begin{samepage}\enlargethispage{3\baselineskip}
so that the regular variation of $\mu$ and the definition of $s_n$ imply \eqref{condExp}.  
Thus, in the upper bound
\begin{align*}
\begin{split}
\PP\left(\left|Y_{\rho_{c,n}}-c^\al\right|\geq\e\, c^\al\right) & \ \loe \
 \PP\left(\left|\E\left[Y_{\rho_{c,n}}\left|N_n\right.\right]-c^\al\right|\geq \frac{\e}{2}\, c^\al\right)\\[1ex]
&\hspace*{100pt} +\ \PP\left(\left|Y_{\rho_{c,n}}-\E\left[Y_{\rho_{c,n}}\left|N_n\right.\right]\right|\geq\frac{\e}{2}\,c^\alpha\right)
\end{split}
\end{align*}
with $\e>0$,  
the first right-hand probability converges to $0$. For the second one, Chebyshev's inequality and Lemma \ref{Lambda} (ii) imply that
 \end{samepage}
\begin{align*}
\PP\big(\big|Y_{\rho_{c,n}}-\E\big[Y_{\rho_{c,n}}\big| N_n\big]\big|\geq\,\e\,c^\alpha\big)  &\eq \E\left[\PP\left(\left|Y_{\rho_{c,n}}-\E\left[Y_{\rho_{c,n}}\left|N_n\right.\right]\right|\geq\e\,c^\alpha\left.\right|N_n\right) \right]\\[1ex]
& \loe  \E\left[\frac{\V\left(Y_{\rho_{c,n}}\left|N_n\right.\right)}{\e^2c^{2\alpha}}\wedge 1\right]\\[1ex]
& \loe  \E\left[\frac{\E\left[Y_{\rho_{c,n}}\left|\right.N_n\right]}{\e^{2}\,c^{2\alpha}}\wedge 1\right].
\end{align*}

From \eqref{condExp} and dominated convergence, we conclude
\[\PP\big(\big|Y_{\rho_{c,n}}-\E\big[Y_{\rho_{c,n}}\big| N_n\big]\big|\geq\,\e\,c^\alpha\big) \loe \e^{-2}c^{-\al}+\oo(1)\]
as $n\rightarrow\infty$, which provides the claim. 
\end{proof}

\bigskip

For the following lemma, let us recall the subdivided external branch lengths 

\[\widecheck{T}^{n}_{i}\ :=\ T^n_{i}\wedge \widetilde{\rho}_{c,n} \quad \text{ and }\quad \widehat{T}^{n}_i\ :=\ T^n_{i}-\widecheck{T}^{n}_{i}\]

for $1\leq i\leq n$ and let 
\[\beta\ := \ \frac{\alpha-1}{\alpha}.\] 

\medskip

\begin{lem}\label{lem:key} 
Suppose that the $\Lambda$-coalescent has a regularly varying rate $\mu$ with exponent $\al\in(1,2]$.   
Then, for $\ell,y\in\N$, there exist random variables $U_{1,y}\geq\ldots\geq U_{\ell,y}$ such that the following convergence results hold: 
\begin{enumerate}
\item For any bounded continuous function $g:\R^\ell\to \R$ and for fixed $y\geq\ell$, as $n\to\infty$,
\[\E\left[g\left(\kappa(cs_n)\widehat{T}^{n}_{\left\langle 1\right\rangle},\ldots,\kappa(cs_n)\widehat{T}^{n}_{\left\langle \ell \right\rangle}\right)\, \big|\,Y_{\rho_{c,n}}=y,X_{\rho_{c,n}}\right]\  \longrightarrow\ \E\left[g \left(U_{1,y},\ldots,U_{\ell,y}\right)\right]\]
in probability. 
\item For fixed $\ell\in\N$, as $y\to\infty$, 
\[y^{-\beta}\left(U_{1,y},\ldots,U_{\ell,y}\right)\ \stackrel{d}{\longrightarrow} \ \left(U_{1},\ldots,U_{\ell}\right),\]
where $U_{1}>\cdots> U_{\ell}$ are the points of the Poisson point process of Theorem \ref{reg_var v2}. 
\end{enumerate}
\end{lem}

\begin{proof}
(i) Let 
\[\widebar{g}_y(x,z) := \E\left[g\left(\kappa(z)\widehat{T}^{n}_{\left\langle 1\right\rangle},\ldots,\kappa(z)\widehat{T}^{n}_{\left\langle \ell \right\rangle}\right)\, \big|\,Y_{\rho_{c,n}}=y,X_{\rho_{c,n}}=x\right]\]
for $x>y, z\ge 2$. 
Observe that due to the strong Markov property, given the events $X_{\rho_{c,n}}=x$ and  $Y_{\rho_{c,n}}=y$, the $y$ remaining external branches evolve as $y$ ordinary external branches out of a sample of $x$ many individuals. From these $y$ external branches, we consider the $\ell$ largest ones. 
Hence, since $\kappa$ is regularly varying, Corollary \ref{cor} yields that 
\[\widebar{g}_y(x,z) \ \longrightarrow \  \E\left[g \left(U_{1,y},\ldots,U_{\ell,y}\right)\right]\]
as $x\to\infty$ and $z/x\to 1$. Here, from established formulae for order statistics of i.i.d random variables,  $(U_{1,y},\ldots,U_{\ell,y})$ has the density 
\begin{align} \label{jointdens}
\ell!\binom{y}{\ell}F\left( u_\ell\right)^{y-\ell}\ \prod_{i=1}^{\ell}\,f\big(u_i\big) du_1\cdots du_\ell,
\end{align}
with $u_1 \geq \cdots \geq u_\ell\geq 0$, 
where $f$ is the density from formula \eqref{dens} and $F$ its cumulative distribution function.   

Now, it follows from Skorohod's representation theorem that one can construct random variables $X_n'$ on a common probability space with the properties that $X_n'$ and $X_{\rho_{c,n}}$ have the same distribution for each $n \ge 1$ and that, in view of Proposition \ref{EV} (ii), the random variables $X_n'/cs_n$ converge to 1 a.s. It follows
\[ \widebar g_y(X_n', cs_n) \to \E\left[g \left(U_{1,y},\ldots,U_{\ell,y}\right)\right]  a.s. \]
and, therefore,
\[ \widebar g_y(X_{\rho_{c,n}}, cs_n) \to \E\left[g \left(U_{1,y},\ldots,U_{\ell,y}\right)\right]   \]
in probability, which is our claim. \bigskip


(ii) Note that 
\[y^{\beta+1}f(y^\beta u) \eq y^{\beta+1}\al\left(1+(\al-1)uy^\beta\right)^{-1-1/\beta} \ \stackrel{y\to\infty}{\longrightarrow} \ \al\left(\left(\al-1\right)u\right)^{-1-1/\beta}\]
and
\[F(y^\beta u)^{y-\ell} \eq \left[1-\left(1+(\al-1)y^\beta u\right)^{-1/\beta}\right]^{y-\ell} \ \stackrel{y\to\infty}{\longrightarrow} \ \exp{\left(-\left((\alpha-1)\,u\right)^{-1/\beta}\right)}. \]
Consequently, 
\[\ell!\binom{y}{\ell}F\left( y^\beta u_\ell\right)^{y-\ell}\ \prod_{i=1}^{\ell}\left[\,f\big(y^\beta u_i\big)y^\beta du_i\right],\]
being the density of $y^{-\beta}\left(U_{1,y},\ldots,U_{\ell,y}\right)$, has the limit 
\[\exp{\left(-\left((\alpha-1)\,u_\ell\right)^{-1/\beta}\right)}\,\prod_{i=1}^{\ell}\, \alpha\left((\alpha-1)\,u_i\right)^{-1-1/\beta}du_1\cdots du_\ell\] 
as $y\to\infty$. 
Indeed, this is the joint density of the rightmost points $U_1>\cdots>U_\ell$ of the Poisson point process given in Theorem \ref{reg_var v2}.  
\end{proof}

\bigskip
\newpage

\begin{proof}[Proof of Theorem \ref{reg_var v2}] 
The proof consists of two parts. First, we consider $(\widehat{T}^n_{\left\langle 1 \right\rangle},\ldots,\widehat{T}^n_{\left\langle \ell \right\rangle})$ in the limits $n\to\infty$ and then $c\to\infty$, which gives already the limit of our theorem.  
Consequently, in the second step it remains to show that $(\widecheck{T}^n_{\left\langle 1 \right\rangle},\ldots,\widecheck{T}^n_{\left\langle \ell \right\rangle})$ can asymptotically be  neglected. 

\medskip

In the first step, we normalize $\widehat{T}^n_{\left\langle j \right\rangle}$ not by  $\kappa(s_n)$ but by the factor $Y_{\rho_{c,n}}^{-\beta}\kappa(c s_n)$, which is equivalent in the limit $c \to \infty$ because of Proposition \ref{EV} (iii).
Thus, we set 
\[V_{c,n} \ :=\ \kappa(cs_n)\left(\widehat{T}^n_{\left\langle 1 \right\rangle},\ldots,\widehat{T}^n_{\left\langle \ell \right\rangle}\right).\]
Let $g:\R^\ell\rightarrow\R$ be a continuous function and assume that $\max{|g|}\leq 1$. 
For $c>0$, we obtain via the law of total expectation and Lemma \ref{lem:key} (i) that
\begin{align*}
\Big|&\E\Big[g\Big(Y_{\rho_{c,n}}^{-\beta}\,V_{c,n}\Big)\,\Big|\,X_{\rho_{c,n}}\Big]\ -\ \E\left[g\left(U_{1},\ldots,U_{\ell}\right)\right]\Big|\\[2ex]
&\loe \sum_{c/2\leq y\leq 2c}\Big|\E\left[\left.g\left(y^{-\beta}\,V_{c,n}\right)\,\right|\,Y_{\rho_{c,n}}=y,\,X_{\rho_{c,n}}\right]\ -\ \E\left[g\left(U_{1},\ldots,U_{\ell}\right)\right]\Big|\hspace{-1pt}\cdot\hspace{-1pt}\PP\left(Y_{\rho_{c,n}}=y\,|\,X_{\rho_{c,n}}\right)\\[1ex]
& \qquad\qquad +\ 2\PP\left(\left|Y_{\rho_{c,n}}-c^\al\right|\geq c^\al/2\,|\,X_{\rho_{c,n}}\right)\\[1.5ex]
&\loe  \max_{c/2\leq y\leq 2c}\Big|\E\left[\left.g\left(y^{-\beta}\,V_{c,n}\right)\,\right|\,Y_{\rho_{c,n}}=y,\,X_{\rho_{c,n}}\right]\ -\ \E\left[g\left(U_{1},\ldots,U_{\ell}\right)\right]\Big| \\[1ex]
& \qquad\qquad +\ 2\PP\left(\left|Y_{\rho_{c,n}}-c^\al\right|\geq c^\al/2\,|\,X_{\rho_{c,n}}\right) \\[1.5ex]
&\loe  \max_{c/2\leq y\leq 2c}\Big|\E\left[g\left(y^{-\beta}U_{1,y},\ldots,y^{-\beta}U_{\ell,y}\right)\right]\ -\ \E\left[g\left(U_{1},\ldots,U_{\ell}\right)\right]\Big|  +\oo_P(1)\\[1ex]
& \qquad\qquad +\ 2\PP\left(\left|Y_{\rho_{c,n}}-c^\al\right|\geq c^\al/2\,|\,X_{\rho_{c,n}}\right)
\end{align*}
as $n\to\infty$. Without loss of generality, we may assume that the $\oo_P(\cdot)$- term is bounded by 1. Hence, taking expectations, applying Jensen's inequality to the left-hand side and using dominated convergence, we obtain 
\begin{align*}
&\Big|\E\Big[g\Big(Y_{\rho_{c,n}}^{-\beta}\,V_{c,n}
\Big)\Big]\ -\ \E\left[g\left(U_{1},\ldots,U_{\ell}\right)\right]\Big|\\[1.5ex]
&\qquad\loe \  \max_{c/2\leq y\leq 2c}\Big|\E\left[g\left(y^{-\beta}U_{1,y},\ldots,y^{-\beta}U_{\ell,y}\right)\right]\ -\ \E\left[g\left(U_{1},\ldots,U_{\ell}\right)\right]\Big|\ +\ \oo(1)\\[1ex]
&\qquad\qquad \qquad +\ 2\PP\left(\left|Y_{\rho_{c,n}}-c^\al\right|\geq c^\al/2\right)
\end{align*}
as $n\to\infty$. 
Then Lemma \ref{lem:key} (ii) and Proposition \ref{EV} (iii) entail
\begin{align}\label{key}
\limsup_{n\to\infty}\Big|\E\Big[g\Big(Y_{\rho_{c,n}}^{-\beta}\,V_{c,n}
\Big)\Big]-\E\left[g\left(U_{1},\ldots,U_{\ell}\right)\right]\Big|\ \stackrel{c\to\infty}{\longrightarrow}\ 0. 
\end{align}

This finishes the first part of our proof. For the second one, we additionally assume that $g$ is a Lipschitz continuous function with Lipschitz constant $1$ (in each coordinate) and prove that
\begin{align}\label{proof}
\E\left[g\left(\kappa(s_n)\,T^n_{\left\langle 1 \right\rangle},\ldots,\kappa(s_n)\,T^n_{\left\langle \ell \right\rangle}\right)\right]\ \stackrel{n\to\infty}{\longrightarrow} \ \E\left[g\left(U_{1},\ldots,U_{\ell}\right)\right],
\end{align}
which implies the theorem's statement.  
For $\e>0$, we have 
\begin{align*}
\big|\E&\big[g\big(\kappa(s_n)\,T^n_{\left\langle 1 \right\rangle},\ldots,\kappa(s_n)\,T^n_{\left\langle \ell \right\rangle}\big)\big]\ -\ \E\big[g\big(U_{1},\ldots,U_{\ell}\big)\big]\big|\\[1ex]
& \loe \left|\E\left[g\left(\kappa(s_n)\,\widehat{T}^n_{\left\langle 1 \right\rangle},\ldots,\kappa(s_n)\,\widehat{T}^n_{\left\langle \ell \right\rangle}\right)\right]\ -\ \E\left[g\left(U_{1},\ldots,U_{\ell}\right)\right]\right|+\sum_{i=1}^\ell\E\left[\kappa(s_n)\widecheck{T}^n_{\left\langle i\right\rangle}\wedge 2\right]\\[1ex]
 &\loe \Big|\E\left[g\left(Y_{\rho_{c,n}}^{-\beta}\,V_{c,n}\right)\right]\ -\ \E\left[g\left(U_{1},\ldots,U_{\ell}\right)\right]\Big|\\[1ex]
&\qquad \quad +\ 
\sum_{i=1}^\ell\E\left[\left|\left(Y_{\rho_{c,n}}^{-\beta}\kappa(cs_n)-\kappa(s_n)\right)\widehat{T}^n_{\left\langle i\right\rangle}\right|\wedge 2\right] +\ell\,\E\left[\kappa(s_n)\widecheck{T}^n_{\left\langle 1\right\rangle}\wedge 2\right]\\[1ex]
&\loe \left|\E\left[g\left(Y_{\rho_{c,n}}^{-\beta}\,V_{c,n}\right)\right]\ -\ \E\left[g\left(U_{1},\ldots,U_{\ell}\right)\right]\right|\\[1ex]
& \qquad\quad + \ \ell\, \E\left[\left(\e\kappa(cs_n)Y_{\rho_{c,n}}^{-\beta}\widehat{T}_{\left\langle 1\right\rangle}\right)\wedge 2\right]+\ 2\ell\,\PP\left(\left|Y_{\rho_{c,n}}^{-\beta}\kappa(cs_n)-\kappa(s_n)\right|\geq\e\kappa(cs_n)Y_{\rho_{c,n}}^{-\beta}\right)\\[1ex]
& \qquad\quad \quad
+\ \ell\e + 2\ell\,\PP\left(\kappa(s_n)\,\widecheck{T}_{\left\langle 1\right\rangle}\geq \e \right)
\end{align*}
and, consequently,
\begin{align*}
&\limsup_{n\to\infty}\big|\E\big[g\big(\kappa(s_n)\,T^n_{\left\langle 1 \right\rangle},\ldots,\kappa(s_n)\,T^n_{\left\langle \ell \right\rangle}\big)\big]-\E\big[g\big(U_{1},\ldots,U_{\ell}\big)\big]\big|\\[1ex]
&\qquad\leq \ \limsup_{n\to\infty}\left|\E\left[g\left(Y_{\rho_{c,n}}^{-\beta}\,V_{c,n}\right)\right]-\E\left[g\left(U_{1},\ldots,U_{\ell}\right)\right]\right|\\[1ex]
&\qquad\qquad\quad +\ \ell\,\limsup_{n\to\infty} \left|\E\left[\left(\e\kappa(cs_n)Y_{\rho_{c,n}}^{-\beta}\widehat{T}_{\left\langle 1\right\rangle}\right)\wedge 2\right]-\E\left[\left(\e U_1\right)\wedge 2\right]\right|+\ell\,\E\left[\left(\e U_1\right)\wedge 2\right]\\[1ex]
&\qquad\qquad\quad\quad  +\ 2\ell\limsup_{n\to\infty}\PP\left(\left|1-\frac{\kappa(s_n)}{\kappa(cs_n)}Y_{\rho_{c,n}}^\beta\right|\geq \e\right)\\[1ex]
&\qquad\qquad\quad\quad\quad +\ \ell\e + 2\ell\,\limsup_{n\to\infty}\PP\left(\kappa(s_n)\,\widetilde{\rho}_{c,n}\geq \e\right).
\end{align*} 
We now use \eqref{key} for the first two right-hand terms and Proposition \ref{EV} (iii) for the first probability taking $\kappa(cs_n)/\kappa(s_n)\sim c^{\al-1}=c^{\al\beta}$ also into account. 
To the other probability, we apply Proposition~\ref{EV}~(i). 
Hence, passing to the limit as $c\to\infty$ yields
\begin{align*}
&\limsup_{n\to\infty}\big|\E\big[g\big(\kappa(s_n)\,T^n_{\left\langle 1 \right\rangle},\ldots,\kappa(s_n)\,T^n_{\left\langle \ell \right\rangle}\big)\big]-\E\big[g\big(U_{1},\ldots,U_{\ell}\big)\big]\big| \loe \ell\,\E\left[\left(\e U_1\right)\wedge 2\right] +\ell\e. 
\end{align*} 
\enlargethispage{2\baselineskip}

Finally, taking the limit $\e\to 0$ and using dominated convergence provides the claim.  
\end{proof}

\bigskip
 
\section{Proof of Theorem \texorpdfstring{\protect\ref{bs}}{1.6}} \label{sec_proof_bs}

Recall the notation of the reversed order statistics $Z_{\left\langle 1\right\rangle}\geq Z_{\left\langle 2\right\rangle}\geq \cdots$ of real-valued random variables as introduced in the previous section and the definition 
\[t_n\ :=\ \log\log{n}-\log\log\log{n}+\log\log\log{n}/\log\log{n}.\] 

In this section, we prove the following equivalent version of Theorem \ref{bs}:  

\begin{theorem} \label{bs v2}
For the Bolthausen-Sznitman coalescent, the following convergence holds: For $\ell\in\N$, 
\[\log{\log{n}}\left(T_{\left\langle 1\right\rangle}^{n}-t_n\,\ldots,T_{\left\langle \ell\right\rangle}^n-t_n\right)\ \overset{d}{\longrightarrow}\ \left(U_1-G,\ldots,U_\ell-G\right)\]
as $n\to\infty$, where $U_1>\cdots> U_\ell$ are the $\ell$ maximal points in decreasing order of a Poisson point process on $\R $ with intensity measure $e^{-x}\;dx$ and $G$ is an independent standard Gumbel distributed random variable. 
\end{theorem} 

\medskip
  
Recall, for $c>1$, the notion
\[t_{c,n}\ :=\ t_n-\frac{\log{c}}{\log\log{n}}.\]  

\medskip 

\begin{lem}\label{N}
Let $E$ be a standard exponential random variable. Then, as $n\to\infty$, we have for $c>1$,  
\[e^{-t_{c,n}}N_n(t_{c,n})\ \overset{d}{\longrightarrow}\ cE.\]
\end{lem}

\begin{proof}
We first consider $N_n(t)^{(r)}:=N_n(t)\left(N_n(t)+1\right)\cdots\left(N_n(t)+r-1\right)$ for $r\in\N $. 
For these ascending factorials, Lemma 3.1 of \cite{Moe15} provides
\[\E\left[N_n(t)^{(r)}\right] \eq \frac{\Gamma\left(r+1\right)}{\Gamma\left(1+re^{-t}\right)}\frac{\Gamma\left(n+re^{-t}\right)}{\Gamma\left(n\right)}.\] 
The Sterling approximation with remainder term yields uniformly in $t\geq 0$,
\[\frac{\Gamma\left(n+re^{-t}\right)}{\Gamma\left(n\right)} \eq n^{re^{-t}}\left(1+\oo\left(1\right)\right)\]
and, consequently,
\[\E\left[N_n(t)^{(r)}\right] \eq \frac{\Gamma\left(r+1\right)}{\Gamma\left(1+re^{-t}\right)}\ n^{re^{-t}}\left(1+\oo\left(1\right)\right)\]
uniformly in $t\geq 0$ as $n\rightarrow\infty$.
Inserting $t_{c,n}$ in this equation entails 
\enlargethispage{\baselineskip}
\begin{align*}
n^{-re^{-t_{c,n}}}\E\left[N_n(t_{c,n})^{(r)}\right]\ \rightarrow\  r!
\end{align*}

as $n\rightarrow\infty$.

 \medskip

Now observe
\begin{align*}
e^{-t_{c,n}}\log{n}&\eq \exp{\left(-\frac{\log\log\log{n}}{\log\log{n}}+\frac{\log{c}}{\log\log{n}}\right)}\log\log{n}\\[1ex]
&\eq\log{\log{n}}-\log{\log{\log{n}}}+\log{c} +\oo\left(1\right)\\[1ex]
&\eq t_{c,n}+\log{c}+\oo(1).
\end{align*}
Equivalently,
\begin{align*}
n^{e^{-t_{c,n}}}\eq ce^{t_{c,n}}\left(1+\oo\left(1\right)\right)
\end{align*}
and, therefore,
\begin{align}\label{asc_fact}
e^{-rt_{c,n}}\E\left[N_n(t_{c,n})^{(r)}\right]\ \rightarrow\  c^rr!
\end{align}
as $n\to\infty$. \medskip

Furthermore, because of  
\[N_n(t)^{r}\loe N_n(t)^{(r)}\loe N_n(t)^r+2^rr^rN_n(t)^{r-1}\loe N_n(t)^r+2^rr^rN_n(t)^{(r-1)},\]
we have
\[N_n(t)^{(r)}-2^rr^rN_n(t)^{(r-1)}\loe N_n(t)^r\loe N_n(t)^{(r)}.\]
Thus, \eqref{asc_fact} transfers to 
\[e^{-rt_{c,n}}\E\left[N_n(t_{c,n})^{r}\right]\ \longrightarrow \ c^rr!\] 
as $n\to\infty$ and our claim follows by method of moments. 
\end{proof}

\bigskip

The following lemma provides the asymptotic behavior of the joint probability distribution of the lengths of the longest external branches starting at time $t_{c,n}$. Let 
\[M_n(t) \ := \ \#\left\{i\geq 1 :\;\{i\}\in\Pi_n(t)\right\},\qquad t\geq 0,\]  
which is the number of external branches at time $t$. Also recall 
\[\widehat{T}^n_{\left\langle i\right\rangle}:=(T_{\left\langle i\right\rangle}^n-t_{c,n})^+.\] 

 \begin{samepage}\enlargethispage{2\baselineskip}
\begin{lem}\label{lem:key2} 
For $\ell,y\in\N$, there exist random variables $U_{1,y}\geq\cdots\geq U_{\ell,y}$ such that the following convergence results hold:
\begin{enumerate}
\item For any bounded continuous function $g:\R^\ell\to \R$ and for fixed natural numbers $\ell\leq y$, as $n\to\infty$,
\begin{align*}
&\E\left[g\left(\log{\log{(n)}}\,\big(\widehat{T}^{n}_{\left\langle 1 \right\rangle},\ldots,\widehat{T}^{n}_{\left\langle \ell \right\rangle}\big)\right)\,\Big|\,N_n(t_{c,n}),M_n(t_{c,n})=y\right] \ \longrightarrow \ \E\left[g\left(U_{1,y},\ldots,U_{\ell,y}\right)\right]
\end{align*}
in probability. 
\item For fixed $\ell$, as $y\to\infty$,
\[ \left(U_{1,y}-\log{y}
,\ldots,U_{\ell,y}-\log{y}
\right) \  \stackrel{d}{\longrightarrow} \ \left(U_{1},\ldots,U_{\ell}\right),\]
where $U_{1}>\cdots> U_{\ell}$ are the points of the Poisson point process of Theorem \ref{bs v2}. 
\end{enumerate}
\end{lem}
\end{samepage}

\begin{proof}
(i) We proceed in the same vein as in the proof of Lemma \ref{lem:key} (i). The strong Markov property, Corollary \ref{cor} (see also formula \eqref{ex} in the first example) and Lemma \ref{N} yield that

\[\E\left[g\left(z\,\big(\widehat{T}^{n}_{\left\langle 1 \right\rangle},\ldots,\widehat{T}^{n}_{\left\langle \ell \right\rangle}\big)\right)\,\Big|\,N_n(t_{c,n})=x,M_n(t_{c,n})=y\right] \ \longrightarrow \ \E\left[g\left(U_{1,y},\ldots,U_{\ell,y}\right)\right]\]
as $x\to\infty$ and $z/\log{x}\to 1$, where 
$(U_{1,y},\ldots,U_{\ell,y})$ has the density  
\begin{align}\label{dens_bs}
\ell!\binom{y}{\ell}\left(1-e^{-u_\ell}\right)^{y-\ell}\,\prod_{i=1}^{\ell}e^{-u_i} du_1\cdots du_\ell
\end{align}
for $u_1\geq\cdots\geq u_\ell$.  
Moreover, from Lemma \ref{N}, we obtain
\begin{align*}
\log{\left(N_n(t_{c,n})\right)} \eq t_{c,n} +\OO_P(1) \eq \log{\log{n}}+\oo_P\left(\log{\log{n}}\right)
\end{align*}
as $n\rightarrow\infty$. Thus, replacing $x$ and $z$ above by $N_n(t_{c,n})$ and $\log\log{n}$, respectively,  and invoking Skorohod's representation theorem once more, our claim follows. \bigskip

(ii) Shifting the distribution from \eqref{dens_bs} by $\log{y}$, we arrive at the densities
\[\ell!\binom{y}{\ell}\left(1-\frac{e^{-u_\ell}}{y}\right)^{y-\ell}\,y^{-\ell}\prod_{i=1}^{\ell}e^{-u_i} du_1\cdots du_\ell\]
and their limit
\[e^{-e^{-u_\ell}}\prod_{i=1}^{\ell}e^{-u_i}du_i\]
as $y\to\infty$, which is the joint density of $U_1,\ldots, U_\ell$. This finishes the proof. 
\end{proof}

\bigskip

Next, we introduce the notion  
\[\rho_{c,n} \ := \ \min{\bigg\{k\geq 1:\ \sum_{j=0}^{k-1} W_j>t_{c,n}\bigg\}}\wedge\tau_n.\]

\medskip

It is important to note that in the case of the Bolthausen-Sznitman coalescent 
Proposition \ref{ErgLem2} is no longer helpful and   
we may not simply apply \eqref{Prop.3}. 
As a substitute, we shall use the following lemma.

\begin{lem}\label{sum_bs}
As $n\to\infty$,
\[\sum_{j=0}^{\rho_{c,n}-1}\frac{1}{X_j} \eq t_{c,n}+\oo_P(1).\]
\end{lem}

\begin{proof} 
Let $\mathcal{F}_k:=\sigma\left(X,W_0,\ldots,W_{k-1}\right)$ and 
\[Z_k\ :=\ \sum_{j=0}^{k\wedge \tau_n-1}\left(W_j-\frac{1}{X_j-1}\right), \qquad k\geq 0.\]
In particular, we have $Z_0=0$. 
Given $\mathcal{F}_j$ and $X_j=b$ with $b\geq 2$, the waiting time $W_j$ in the Bolthausen-Sznitman coalescent is exponential with rate parameter $b-1$ (see (47) in \cite{Pit99}).  
Thus, $(Z_k)_{k\in\N}$ is a martingale with respect to the filtration $(\mathcal{F}_k)_{k\in\N}$ with (predictable) quadratic variation 
\[\langle Z\rangle_{k} \ :=\ \sum_{j=0}^{k\wedge \tau_n-1}\E\left[(Z_{j+1}-Z_j)^2\big|\mathcal{F}_j\right] \eq \sum_{j=0}^{k\wedge\tau_n-1} \frac{1}{(X_j-1)^2} \qquad a.s.\] 
Applying Doob's optional sampling theorem to the martingale $Z_{k}^2-\langle Z\rangle_{k}$ yields
\begin{align} \label{3/X}
\E\left[Z^2_{\rho_{c,n}}\right] \eq \E\left[\left\langle Z\right\rangle_{\rho_{c,n}}\right]\eq \E\left[\sum_{j=0}^{\rho_{c,n}-1}\frac{1}{(X_j-1)^2}\right]\loe 
\E\left[\sum_{k=X_{\rho_{c,n}-1}}^\infty \frac{1}{(k-1)^{2}}\right] 
\end{align} 
and, therefore, because of $X_{\rho_{c,n}-1}=N_n(t_{c,n})$ a.s.,
\[\E\left[Z^2_{\rho_{c,n}}\right]\loe \E\left[\frac{4}{N_n(t_{c,n})}\right].\]
By Lemma \ref{N} and dominated convergence, the right-hand term converges to $0$ as $n\to\infty$ implying  
\[\sum_{j=0}^{\rho_{c,n}-1}\left(W_j-\frac{1}{X_j}\right)\eq Z_{\rho_{c,n}} +\OO_P\left(\frac{4}{X_{\rho_{c,n}-1}}\right)\eq\oo_P(1)\]
as $n\to\infty$. 
\enlargethispage{2\baselineskip}
Finally, the quantity 
\(\sum_{j=0}^{\rho_{c,n}-1}W_j-t_{c,n}\) 
is the residual time the process $N_n$ spends in the state $N_n(t_{c,n})$. Due to the property that exponential times lack memory, the residual time is exponential with parameter $N_n(t_{c,n})$. Thus, in view of Lemma \ref{N}, the residual time converges to $0$ in probability. This finishes the proof.    
\end{proof}

\bigskip

\begin{lem}\label{M} For the number of external branches at time $t_{c,n}$, we have the following results:
\begin{enumerate}
\item For $c>1$,  
\[\E\left[M_n(t_{c,n})\left.\right|N_n\right]\ \overset{d}{\longrightarrow}\ c\,E\]
as $n\to\infty$, where $E$ denotes a standard exponential random variable.
\item For $\e>0$, as $c\to\infty,$
\[\limsup_{n\to\infty} \PP\left(\left|M_n(t_{c,n})\ -\ \E\left[M_n(t_{c,n})\left|N_n\right.\right]\right|>c^{1/2+\e}\right)\ \rightarrow \ 0 \]
as well as 
\[\limsup_{n\to\infty}\PP\left(M_n(t_{c,n})> c^{1+\e}\right) \ \rightarrow \ 0 \qquad \text{ and } \qquad \limsup_{n\to\infty}\PP\left(M_n(t_{c,n})<c^{1-\e}\right)\ \rightarrow \ 0. \] 
\end{enumerate}
\end{lem}

\begin{proof}
(i) Using the representation from Lemma \ref{Lambda} (i) and a Taylor expansion as in \eqref{Mac}, we get  
\begin{align*}
\E\left[Y_{\rho_{c,n}-1}\left.\right|N_n\right]  \eq X_{\rho_{c,n}-1} \exp{\left(-\sum_{j=1}^{\rho_{c,n}-1}\frac{1}{X_j}+\OO_P\left(X_{\rho_{c,n}-1}^{-1}\right)\right)}
\end{align*}
as $n\to\infty$. 
Recall that the definition of $\rho_{c,n}$ entails $N_n(t_{c,n})=X_{\rho_{c,n}-1}$ and $M_n(t_{c,n})=Y_{\rho_{c,n}-1}$ a.s. 
Thus, we obtain
\begin{align}\label{EW_Y}
\E\left[M_n(t_{c,n})\left.\right|N_n\right]   \eq N_n(t_{c,n}) \exp{\left(-\sum_{j=1}^{\rho_{c,n}-1}\frac{1}{X_j}+\OO_P\left(N_n(t_{c,n})^{-1}\right)\right)}.
\end{align}
From Lemma \ref{sum_bs} and Lemma \ref{N}, it follows
\begin{align*}
\E\left[M_n(t_{c,n})\left.\right|N_n\right]   \eq N_n(t_{c,n}) \exp{\left(-t_{c,n}+\oo_P\left(1\right)\right)}.
\end{align*}
Hence, Lemma \ref{N} implies our claim. 

\bigskip

\enlargethispage{\baselineskip}
(ii) Chebyshev's inequality and Lemma \ref{Lambda} (ii) provide
\begin{align*}
&\PP\left(\left|M_n(t_{c,n})-\E\left[M_n(t_{c,n})\left|N_n\right.\right]\right|>c^{1/2+\e}\right)\\[1ex]
& \hspace{6pc}\eq \E\left[\PP\left(\left|M_n(t_{c,n})-\E\left[M_n(t_{c,n})\left|N_n\right.\right]\right|>c^{1/2+\e}\big|N_n\right)\right]\\[1ex]
&\hspace{6pc}\loe \E\left[\frac{\V\left(M_n(t_{c,n})\left.\right|N_n\right)}{c^{1+2\e}}\wedge 1\right]\\[1ex]
&\hspace{6pc}\loe \E\left[\frac{\E\left(M_n(t_{c,n})\left.\right|N_n\right)}{c^{1+2\e}}\wedge 1\right]. 
\end{align*}
From statement (i) it follows that
\begin{equation*}
\limsup_{n\to\infty}\PP\left(\left|M_n(t_{c,n})-\E\left[M_n(t_{c,n})\left|N_n\right.\right]\right|>c^{1/2+\e}\right) \loe \E\left[\frac{cE}{c^{1+2\e}}\wedge 1\right] \loe c^{-2\e}\,,
\end{equation*}
which entails the first claim. \medskip

Similarly, Markov's inequality yields  
\begin{align*}
\limsup_{n\to\infty}\PP\left(M_n(t_{c,n})> c^{1+\e}\right)\loe \limsup_{n\to\infty}\E\left[\frac{\E\left[M_n(t_{c,n})\left|N_n\right.\right]}{c^{1+\e}}\wedge 1\right] \loe c^{-\e}
\end{align*}
giving the second claim. \medskip

Furthermore, we have
\begin{align*}
\PP\big(M_n(t_{c,n})< c^{1-\e}\big)  \loe \ &  \PP\left(\E\left[M_n(t_{c,n})\left|N_n\right.\right]<2 c^{1-\e} \right) \\[1ex] 
& \hspace{6pc} + \ \PP\left(\left|M_n(t_{c,n})-\E\left[M_n(t_{c,n})\left|N_n\right.\right]\right|> c^{1-\e}\right)
\end{align*}
and, consequently, in view of part (i),
\[\limsup_{n\to\infty}\PP\big(M_n(t_{c,n})< c^{1-\e}\big)  \;\leq\;   \PP\left(E<2c^{-\e} \right)  +\,  \limsup_{n\to\infty}\PP\left(\left|M_n(t_{c,n})-\E\left[M_n(t_{c,n})\left|N_n\right.\right]\right|> c^{1-\e}\right).\] 
The first right-hand term converges to $0$ as $c\to\infty$. Also, as we may assume $\e<1/2$, the second term goes to $0$ in view of the first claim of part (ii). 
\end{proof}

\bigskip

With these preparations, we now turn to the proof of Theorem \ref{bs v2}. 

\medskip

\begin{proof}[Proof of Theorem \ref{bs v2}]
The strategy of this proof resembles that of Theorem \ref{reg_var v2}. However, additional care is required to separate the impact of the parts $\widecheck{T}_i^n$ and $\widehat{T}_i^n$. For this purpose, we consider the functions  
\[g(x_1,\ldots,x_\ell)\ :=\ \exp{\left(i\left(\theta_1 x_1+\cdots+\theta_\ell x_\ell\right)\right)} \qquad \text{ and } \qquad h(x)\ :=\ \exp{\left(i\left(\theta_1+\cdots+\theta_\ell\right)x\right)},\]
where $\theta_i\in\R$ for $1\leq i\leq n$. It is sufficient to prove 
\[\E\left[g\left(\log{\log{(n)}}\left(T_{\left\langle 1\right\rangle}^n-t_n\right),\ldots,\log{\log{(n)}}\left(T_{\left\langle \ell\right\rangle}^n-t_n\right)\right)\right]\ \longrightarrow \  \E\left[g\left(U_1-G,\ldots,U_\ell-G\right)\right]\]
as $n\to\infty$. 
We bound the difference of the terms on both sides. 
Recalling 
\[t_n \eq t_{c,n}+\frac{\log{c}}{\log{\log{n}}},\]
we see that, on the event $\left\{M_n(t_{c,n})\geq\ell\right\}$, it holds 
$T_{\left\langle i\right\rangle}^n\eq\widehat{T}_{\left\langle i\right\rangle}^n+t_{c,n}$ 
and, therefore,
\begin{align}\label{newnew}
\log\log{(n)}\big(T^n_{\left\langle j \right\rangle}-t_n\big)\eq
\left(\log\log{(n)}\,\widehat{T}^n_{\left\langle j \right\rangle}-\log{M_n(t_{c,n})}\right)+\log{\frac{M_n(t_{c,n})}{c}}
\end{align}
for $1\leq j\leq\ell$. 
In conjunction with the independence of $\left(U_1,\ldots,U_\ell\right)$ and the Gumbel random variable $G$, it follows that
\begin{align}\label{split}\nonumber
&\left|\E\left[g\left(\log{\log{(n)}}\left(T_{\left\langle 1\right\rangle}^n-t_n\right),\ldots,\log{\log{(n)}}\left(T_{\left\langle \ell\right\rangle}^n-t_n\right)\right)\right]\, -\, \E\left[g\left(U_1-G,\ldots,U_\ell-G\right)\right]\right| \\[2ex] 
&\qquad\loe
\left|\E\left[g\left(V_{c,n}\right)h\left(\log{\frac{M_n(t_{c,n})}{c}}\right)\right]\, -\, \E\left[g\left(U_1,\ldots,U_\ell\right)\right]\E\left[h\left(-G\right)\right]\right|\\[1ex]\nonumber
&\qquad\quad\qquad + \ 2\,\PP\left(M_n(t_{c,n})<\ell\right),
\end{align}
where, in view of \eqref{newnew}, we now set 
\[V_{c,n}\ :=\ \left(\log\log{(n)}\, \widehat{T}^n_{\left\langle 1 \right\rangle}-\log{M_n(t_{c,n})},\ldots,\log\log{(n)}\, \widehat{T}^n_{\left\langle \ell \right\rangle}-\log{M_n(t_{c,n})}\right).\]
Let us estimate the first term on the right-hand side of \eqref{split}.  
 We have 
\begin{align*}
 \Big|\E\Big[g&\left(V_{c,n}\right)\,h\left(\log{\frac{M_n(t_{c,n})}{c}}\right)\Big]\  -\ \E\left[g\left(U_1,\ldots,U_\ell\right)\right]\E\left[h\left(-G\right)\right]\Big|\\[2ex]
\loe &\  \left|\E\left[g\left(V_{c,n}\right)h\left(\log{\frac{M_n(t_{c,n})}{c}}\right)\right]\ -\ \E\left[g\left(U_1,\ldots,U_\ell\right)\right]\E\left[h\left(\log{\frac{M_n(t_{c,n})}{c}}\right)\right]\right|\\[2ex]
 &\qquad +\ \left|\E\left[ h\left(\log{\frac{M_n(t_{c,n})}{c}}\right)\right]\ -\ \E\left[ h\left(\log{\frac{\E\left[ M_n(t_{c,n})\left|N_n\right.\right]}{c}}\right)\right] \right|\\[2ex]
&\qquad +\ \left|\E\left[ h\left( \log{\frac{\E\left[M_n(t_{c,n})\left|N_n\right.\right]}{c}}\right)\right] \ -\ \E\left[ h\left(-G\right)\right] \right|\\[2ex]
& =: \Delta'_{c,n}+\Delta''_{c,n}+\Delta'''_{c,n} \qquad \text{(say).}
\end{align*}

We bound $\Delta'_{c,n}, \Delta''_{c,n}$ and $\Delta'''_{c,n}$ separately.  
For $\Delta'_{c,n}$, we first consider conditional expectations. For $c>1$, we have, by means of Lemma \ref{lem:key2} (i) in the last step,

\newpage
\begin{align*}
&\bigg|\E\bigg[g\left(V_{c,n}\right)h\left(\log{\frac{M_n(t_{c,n})}{c}}\right)\,\bigg|\,N_n(t_{c,n})\bigg]\ -\ \E\left[g\left(U_1,\ldots,U_\ell\right)\right]\E\left[h\left(\log{\frac{M_n(t_{c,n})}{c}}\right)\,\bigg|\,N_n(t_{c,n})\right]\bigg|\\[2.5ex] 
& \leq  \sum_{\sqrt{c}\leq y\leq c^2}\left|\Big(\E\left[\left.g\left(V_{c,n}\right)\,\right|\,N_n(t_{c,n}),\,M_n(t_{c,n})=y\right] \ -\ \E\left[g\left(U_1,\ldots,U_\ell\right)\right]\Big)h\left(\log{\frac{y}{c}}\right)\right|\\
& \quad\qquad\qquad \cdot\PP\left(M_n(t_{c,n})=y\,\big|\,N_n(t_{c,n})\right) \\[1.5ex] 
&\quad\qquad +\ 2\,\PP\left(M_n(t_{c,n})<\sqrt{c}\,\big|\,N_n(t_{c,n})\right) +\ 2\,\PP\left(M_n(t_{c,n})>c^2\,\big|\,N_n(t_{c,n})\right)\\[2ex]
&   \leq   \max_{\sqrt{c}\leq y\leq c^2}\big|\E\left[\left.g\left(V_{c,n}\right)\,\right|\,N_n(t_{c,n}),\,M_n(t_{c,n})=y\right]\ -\ \E\left[g\left(U_1,\ldots,U_\ell\right)\right]\big|\\[1.5ex]
& \quad \qquad +\ 2\,\PP\left(M_n(t_{c,n})<\sqrt{c}\,|\,N_n(t_{c,n})\right)\ +\ 2\,\PP\left(M_n(t_{c,n})>c^2\,|\,N_n(t_{c,n})\right)\\[2ex]
&   \leq   \max_{\sqrt{c}\leq y\leq c^2}\big|\E\left[g\left(U_{1,y}-\log{y},\ldots,U_{\ell,y}-\log{y}\right)\right]\ -\ \E\left[g\left(U_1,\ldots,U_\ell\right)\right]\big|+\oo_P(1)\\[1.5ex]
& \quad \qquad +\ 2\,\PP\left(M_n(t_{c,n})<\sqrt{c}\,|\,N_n(t_{c,n})\right)\ +\ 2\,\PP\left(M_n(t_{c,n})>c^2\,|\,N_n(t_{c,n})\right)
\end{align*}
as $n\to\infty$. Without loss of generality, we may assume that the right-hand $\oo_P(\cdot)$-term is bounded by 1. Hence, taking expectations, we obtain via dominated convergence 
\begin{align}
\nonumber
\Delta'_{c,n}&  \loe   \max_{\sqrt{c}\leq y\leq c^2}\big|\E\left[g\left(U_{1,y}-\log{y},\ldots,U_{\ell,y}-\log{y}\right)\right]\ -\ \E\left[g\left(U_1,\ldots,U_\ell\right)\right]\big|\ +\ \oo(1)\\[1.5ex]\nonumber
& \quad\qquad \qquad +\ 2\,\PP\left(M_n(t_{c,n})<\sqrt{c})\right)\ +\ 2\,\PP\left(M_n(t_{c,n})>c^2)\right). 
\end{align}
Second, observe that the function $h(\log x)$ is Lipschitz on the interval $[c^{-1/4},\infty)$ with Lipschitz constant $|\theta_1+\cdots+\theta_\ell|c^{1/4}$. Thus,  
\begin{align}\label{Lipschitz}
\nonumber
\Delta''_{c,n} 
\loe &\ \bigg|\E\bigg[h\left(\log{\frac{M_n(t_{c,n})}{c}}\right)\ -\ h\left(\log{\frac{\E\left[ M_n(t_{c,n})\left|N_n\right.\right] }{c}}\right)\,;\,M_{t_{c,n}}\wedge\E\left[ M_n(t_{c,n})\left|N_n\right.\right]\geq c^{3/4}\bigg] \bigg|\\[1.5ex]\nonumber
& \qquad +\ 2\,\PP\left(M_n(t_{c,n})< c^{3/4} \right)\ +\  2\,\PP\left(\E\left[ M_n(t_{c,n})\left|N_n\right.\right] <  c^{3/4} \right)\\[2ex]
 \loe &\  2\,\PP\left(\left|M_n(t_{c,n})\ -\ \E\left[\left.M_n(t_{c,n})\right|N_n\right]\right|>  c^{2/3} \right)\ +\ \left|\theta_1+\cdots+\theta_\ell\right|c^{1/4-1/3} \\[2ex]\nonumber
&\qquad +\ 2\,\PP\left(M_n(t_{c,n}) < c^{3/4} \right)\ +\  2\,\PP\left(\E\left[ M_n(t_{c,n})\left|N_n\right.\right] < c^{3/4} \right).
\end{align}
Last, Lemma \ref{M} (i) provides the convergence of $\Delta'''_{c,n}$ to $0$ as $n\to\infty$. Consequently, combining equation \eqref{split} to \eqref{Lipschitz}, using Lemma \ref{M} and grouping terms yield 
\begin{align*}
&\limsup_{n\to\infty}\bigg|\E\bigg[g\left(V_{c,n}\right)h\left(\log{\frac{M_n(t_{c,n})}{c}}\right)\bigg]\ -\ \E\left[g\left(U_1,\ldots,U_\ell\right)\right]\E\left[h\left(-G\right)\right]\bigg|\\[1.5ex]
& \leq \max_{\sqrt{c}\leq y\leq c^2}\big|\E\left[g\left(U_{1,y}-\log{y},\ldots,U_{\ell,y}-\log{y}\right)\right]\ -\ \E\left[g\left(U_1,\ldots,U_\ell\right)\right]\big|\ \\[1.5ex]
&\qquad + 2\,\limsup_{n\to\infty}\PP\left(M_n(t_{c,n})<\ell\right) + 2\,\limsup_{n\to\infty}\PP\left(M_n(t_{c,n})<\sqrt{c}\right) + 2\,\limsup_{n\to\infty}\PP\left(M_n(t_{c,n})<c^{3/4}\right) \\[1.5ex]
& \qquad \qquad + 2\,\limsup_{n\to\infty}\PP\left(M_n(t_{c,n})>c^2\right) + 2 \limsup_{n\to\infty} \PP\left(\left|M_n(t_{c,n}) - \E\left[\left.M_n(t_{c,n})\right|N_n\right]\right|> c^{2/3}\right)  \\[1.5ex]
&\qquad \qquad\qquad + 2\left(1-e^{-c^{-1/4}}\right)  + \left|\theta_1+\cdots+\theta_\ell\right|c^{-1/12}.
\end{align*}

Finally, taking the limit $c\to\infty$, the right-hand terms converge to $0$ in view of Lemma \ref{lem:key2} (ii) and Lemma \ref{M}.  This finishes the proof. 
\end{proof}

\newpage

\phantomsection

\vfill

\textbf{Acknowledgments.} We are grateful to the anonymous referees for their insightful comments, which allowed us to improve the paper's presentation.

\end{document}